\documentclass[12pt]{article}
\usepackage{amssymb}
\usepackage{latexsym}
\usepackage{amsthm}
\usepackage{amscd}
\usepackage{amsmath}
\usepackage{diagrams}

\newtheorem{thm}{Theorem}[section]
\newtheorem{lem}[thm]{Lemma}
\newtheorem{cor}[thm]{Corollary}
\newtheorem{defn-lem}[thm]{Definition-Lemma}
\newtheorem{conj}[thm]{Conjecture}

\newtheorem{prop}[thm]{Proposition}

\theoremstyle{remark}

\theoremstyle{definition}

\numberwithin{equation}{section}

\allowdisplaybreaks[4]
\def \CA{{\mathcal A}}

\def \C{{\mathbb C}}

\def\map#1.#2.{#1 \longrightarrow #2}
\def\rmap#1.#2.{#1 \dasharrow #2}

\DeclareMathOperator{\corank}{corank}

\def\fb#1.{\underset #1 \to \times}
\def\pr#1.{\Bbb P^{#1}}
\def\ring#1.{\mathcal O_{#1}}
\def\mlist#1.#2.{{#1}_1,{#1}_2,\dots,{#1}_{#2}}

\def\mod{\operatorname{mod}}

\def\uloopr#1{\ar@'{@+{[0,0]+(-4,5)} @+{[0,0]+(0,10)}
@+{[0,0]+(4,5)}}^{#1}}

\def\dloopr#1{\ar@'{@+{[0,0]+(-4,-5)} @+{[0,0]+(0,-10)}
@+{[0,0]+(4,-5)}}_{#1}}

\def\rloopd#1{\ar@'{@+{[0,0]+(5,4)} @+{[0,0]+(10,0)}
@+{[0,0]+(5,-4)}}^{#1}}

\newcommand{\rdots}{\mathinner{\mkern1mu\raise1pt\hbox{.}\mkern2mu\raise4pt\hbox{.} \mkern2mu\raise7pt\vbox{\kern7pt\hbox{.}}\mkern1mu}} 

\def\lloopd#1{\ar@'{@+{[0,0]+(-5,4)} @+{[0,0]+(-10,0)}
@+{[0,0]+(-5,-4)}}_{#1}}

\long\def\ignore#1{}
\long\def\ignore#1{#1}

\begin{document}
\begin{center}
{\bf On the Thom-Boardman Symbols for Polynomial Multiplication Maps}

\bigskip

\end{center}
\begin{center}  
{{\small Jiayuan Lin, Janice Wethington}}
\end{center}

{\small {\bf Abstract} The Thom-Boardman symbol was first introduced by Thom in $1956$ to classify singularities of differentiable maps. It was later generalized by Boardman to a more general setting. Although the Thom-Boardman symbol is realized by a sequence of non-increasing, nonnegative integers, to compute those numbers is, in general, extremely difficult. In the case of polynomial multiplication maps, Robert Varley conjectured that computing the Thom-Boardman symbol for polynomial multiplication reduces to computing the successive quotients and remainders for the Euclidean algorithm applied to the degrees of the two polynomials. In this paper, we confirm this conjecture.}

\section{Introduction}

This paper proves Robert Varley's conjecture on the Thom-Boardman symbols for polynomial multiplication maps.

In 1956, R. Thom developed a method to classify singularities of differentiable maps according to the rank of the first differential of the map and the ranks of its restrictions to submanifolds of singularities. His theory depended upon the manifold structure of the singular locus of each restriction of the map. Eleven years later, J.M. Boardman $[3]$ generalized Thom's work to include maps whose singular loci may fail to be manifolds, or whose successive restrictions may fail to be manifolds. In effect, Boardman expanded Thom's work to almost all differential maps on manifolds. The Thom-Boardman classification is realized by an infinite, non-increasing sequence of nonnegative integers referred to as the Thom-Boardman symbol. When the number of nonzero terms is finite, the sequence for the symbol is usually truncated after the last nonzero entry.

Joint work concerning invariants of Gauss maps of theta divisors by M. Adams, C. McCrory, T. Shifrin and R. Varley $[1]$ revealed a fundamental connection between these Gauss maps and secant maps. Continued work by R. Varley indicated a connection between secant maps and maps defined by the multiplication of two monic single-variable polynomials. The multiplication maps take the coefficients of two polynomials to expressions in those coefficients that describe the coefficients of the product of the two polynomials. The classification by singularities of these polynomial multiplication maps would result in the classification of the secant maps. However, The Thom-Boardman symbol is usually difficult to compute. Even in the case of the polynomial multiplication maps the computation become extremely difficult in all but a small number of cases. A conversation with Victor Goryunov led Robert Varley to conjecture that computing the Thom-Boardman symbol for polynomial multiplication reduces to computing the successive quotients and remainders for the Euclidean algorithm applied to the degrees of the two polynomials.

In her Ph.D. dissertation $[4]$, Janice Wethington revealed many fundamental structures in the Jacobian matrices. She proved Varley's Conjecture in several special cases and obtained upper bounds for the Thom-Boardman symbols. In this paper, we completely prove Varley's Conjecture.

For the reader's convenience, let us first recall the definition of Thom-Boardman symbol from $[2]$ and then state Robert Varley's conjecture.

Let $J$ be an ideal in the algebra $\CA$ of germs at a given point of $C^{\infty}$ maps of manifolds $F: M \rightarrow N, F=(f_1,f_2, \cdots, f_n)$, where $M$ and $N$ have dimensions $m$ and $n$ respectively. Take $x_1, \cdots, x_m$ to be local coordinates on $M$. The $\textit{Jacobian extension}$, $\Delta_k J$, is the ideal spanned by $J$ and all the minors of order $k$ of the $\textit{Jacobian matrix}$ $(\partial f_i/ \partial x_j)$, denoted $\delta J$, formed from partial derivatives of functions $f$ in $J$. Since the determinant of this matrix is multilinear and since $(\partial f/ \partial x')=\partial f/ \partial x \cdot \partial x/ \partial x'$, the $\textit{Jacobian extension}$ is independent of the coordinate system chosen, hence is an invariant of the ideal. We say that $\Delta_i J$ is $\textit{critical}$ if $\Delta_i J \ne \CA$ but $\Delta_{i-1} J = \CA$. That is, the $\textit{critical extension}$ of $J$ is $J$ adjoined with the least order minors of the Jacobian matrix of $J$ for which the extension does not coincide with the whole algebra. If every size minor of $\delta J$ is a unit in $\CA$, then the map was of full rank at the given point already and the critical extension is the ideal $J$ itself. Note that $J \subseteq \Delta_i J$.

Now we shift the lower indices to upper indices of the critical extensions by the rule $\Delta^i J=\Delta_{m-i+1} J$. We repeat the process described above with the resulting ideals until we have a sequence of critical extensions of $J$,

$$J \subseteq \Delta^{i_1} J \subseteq \Delta^{i_2} \Delta^{i_1} J \subseteq \cdots \subseteq \Delta^{i_k} \Delta^{i_{k-1}} \cdots \Delta^{i_1} J= \bf{m}$$

\noindent where $\bf{m}$ is the maximal ideal of $\CA$. Then the $\text{\bf{Thom-Boardman symbol}}$, $TB(J)$, is given by $(i_1, i_2, \cdots, i_k)$. The purpose of switching the indices is that doing so allows us to express $TB(J)$ as follows:

$$i_1=\corank (J), i_2= \corank (\Delta^{i_1} J), \cdots, i_k= \corank (\Delta^{i_{k-1}} \cdots \Delta^{i_1} J)$$

\noindent where the rank of ideal is defined to be the maximal number of independent coordinates from the ideal and the corank is the number of variables minus the rank.

Let $M_n$ be the set of monic complex polynomials in one variable of degree $n$. $M_n \cong \C^n$ by the map sending 
$f(x)=x^n+a_{n-1} x^{n-1}+ \cdots + a_0$ to the $n$-tuple $(a_0,a_1,\cdots, a_{n-1}) \in \C^n$.

If we take $f(x)$ of degree $n$ as above and $g(x)=x^r+ b_{r-1} x^{r-1} + \cdots + b_0$ of degree $r$, then the product $h(x)=f(x) g(x)$ is a monic polynomial of the form $h(x)=x^{n+r} + c_{n+r-1} x^{n+r-1} + \cdots + c_0$, where the $c_j$'s are polynomials in the coefficients of $f$ and $g$. We can also assume that $n \ge r$. The $c_j$'s are as shown below:

$$\left\{\begin{array}{l} 
       c_{n+r-1}=a_{n-1} + b_{r-1} \\ 
       c_{n+r-2}=a_{n-2} + b_{r-2}+ a_{n-1} b_{r-1} \\ 
       and\\
       c_{n+r-j}=a_{n-j} +b_{r-j} + \underset{i+k=n+r-j}{\sum} a_i b_k, \hskip .1 cm \text{for} \hskip .1 cm j \le r \\
       c_{n+r-j}=a_{n-j} + \underset{i+k=n+r-j}{\sum} a_i b_k, \hskip .1 cm \text{for} \hskip .1 cm r< j \le n \\ 
c_{n+r-j}= \underset{i+k=n+r-j}{\sum} a_i b_k, \hskip .1 cm \text{for} \hskip .1 cm j >n\end{array} \right. $$

This gives us maps

$$\mu_{n,r}: \C^n \times \C^r \rightarrow \C^{n+r}$$

\noindent defined by 

$$(a_0, \cdots, a_{n-1}, b_0, \cdots, b_{r-1}) \rightarrow (c_{n+r-1}, \cdots, c_0).$$

\noindent Consider the Euclidean algorithm applied to $n$ and $r$:

$$\begin{array}{lr} n=q_1 r+r_1,  \hskip .44 cm 0 < r_1 < r\\
r=q_2 r_1+r_2,  \hskip .35 cm 0 < r_2 < r_1 \\
\vdots\\
r_{k-1}=q_{k+1} r_k, \hskip .3 cm  0 < r_k <r_{k-1}\end{array} $$

\noindent Let $I(n,r)$ be the tuple given by the Euclidean algorithm on $n$ and $r$:

$$I(n,r)=(r, \cdots, r, r_1, \cdots, r_1, \cdots, r_k, \cdots, r_k)$$

\noindent where $r$ is repeated $q_1$ times, and $r_i$ is repeated $q_{i+1}$ times.

\medskip

Let $I(\mu_{n,r})$ be the ideal in the algebra $\CA$ of germs at origin generated by $c_{j}$'s in the map $\mu_{n,r}: \C^n \times \C^r \rightarrow \C^{n+r}$. Denote $TB(I(\mu_{n,r}))$ the Thom-Boardman symbol of this ideal, Robert Varley conjectured that

\begin{conj} (Varley's Conjecture) $TB(I(\mu_{n,r}))=I(n,r)$ for any $n \ge r$.
\end{conj}

In this paper, we prove Varley's Conjecture, that is, we have 

\begin{thm} $TB(I(\mu_{n,r}))=I(n,r)$ is true for any $n \ge r$.
\end{thm}

One of the difficulties in computing Thom-Boardman symbol is that if we simply add all $(n+r-i_j+1)$ minors into the ideal representing the $j$-th critical extension of $I(\mu_{n,r})$ the number of generators grows exponentially. In her dissertation, Wethington confirmed Varley's Conjecture for all cases $n+r \le 10$ by computer. The memory demands grew exponentially for those calculations. In this paper, we overcome this difficulty by carefully choosing the generators at each step of the critical extensions. Specifically, we find a group of polynomials such that at each step of the critical extensions we only need to add the same number of polynomials indexed by the corresponding entry in $I(n,r)$. We construct these polynomials explicitly and prove that they have the desired property.

This paper is organized as follows: in section $2$ we discuss the first critical extension of $I(\mu_{n,r})$ and prove Varley's Conjecture in the special case $n=r$. In section $3$, we first prove some properties of lower Toeplitz matrices and then construct $(q_1r+r_1)$ polynomials $\psi_0, \cdots, \psi_{q_1r-1}; \psi_{q_1r}, \cdots, \psi_{q_1r+r_1-1}$ explicitly. We show that the $s$-th critical extension of $I(\mu_{n,r})$ is exactly obtained from the previous one by adjoining $\psi_{(s-1)r}, \cdots, \psi_{sr-1}$ for $1 \le s \le q_1$ and the $(q_1+1)$-th critical extension is the $q_1$-th one adjoining $\psi_{q_1r+1}, \cdots, \psi_{q_1r+r_1-1}$. Denote $f_0(x)=f(x)$, $f_1(x)=g(x)$, $r_1=n$ and $r_0=r$. Starting from $f_0(x)$ and $f_1(x)$, we construct a sequence of polynomials $f_2(x), \cdots, f_{k+1}(x)$ inductively such that the degree of $f_i(x)$ is $r_{i-1}$ and the multiplication of $f_{i}(x)$ and $f_{i+1}(x)$ gives a map $\mu_{r_{i-1},r_i}: \C^{r_{i-1}} \times \C^{r_i} \rightarrow \C^{r_{i-1}+r_i}$. Following the same idea we can produce $(q_{i+1}r_i+r_{i+1})$ polynomials with the property that at each of the next $(q_{i+1}+1)$ steps of the critical extensions of $I(\mu_{n,r})$ we only need to add the same number of polynomials indexed by the entries $(r_i, \cdots, r_i, r_{i+1})$ in $I(n,r)$. After adding all such polynomials into $I(n,r)$, we reach the maximal ideal $\bf{m}$. Therefore the rest of the entries in $TB(I(\mu_{n,r}))$ are zeros and Varley's Conjecture follows. 

{\bf Acknowledgment.} We appreciate Professor Robert Varley for his detailed explanation about the motivation to  compute Thom-Boardman symbols of polynomial multiplication maps. Without his help, this collaboration would have never happened.

\section{The First Critical Extension of $I(\mu_{n,r})$}

Let $I(\mu_{n,r})$ be the ideal generated by $c_{n+r-1}, c_{n+r-2}, \cdots, c_0$ defined by the multiplication map $\mu_{n,r}$. There is an interesting fact that becomes obvious when taking the Jacobian $\delta I(\mu_{n,r})$. Taking the derivatives of $c_j$'s in descending order from $n+r-1$ to $0$ with respect to the $a_i$'s and $b_i$'s in descending order from $n-1$ to $0$ and $r-1$ to $0$ respectively, we get the following:  

{\small
\begin{equation}
\begin{split}
\delta I(\mu_{n,r})= 
\begin{pmatrix} 
1 & 0&0& \cdots& \cdots& 0 & 1 & 0& \cdots & 0\cr
b_{r-1} & 1 & 0  &\cdots &\cdots&  0 & a_{n-1} &1&\cdots&0\cr
b_{r-2}& b_{r-1} & 1 & \cdots&\cdots& \cdots & a_{n-2}&a_{n-1}&\cdots&0\cr
\vdots & b_{r-2} & b_{r-1} &\cdots& \cdots&\cdots & \vdots& a_{n-2}&\cdots&1\cr
\vdots &\vdots& b_{r-2}&\cdots&\cdots&1&\vdots&\vdots&\cdots&a_{n-1}\cr
\vdots &\vdots& \vdots&\cdots&\cdots& b_{r-1}&\vdots&\vdots&\cdots&a_{n-2}\cr
b_0 &\vdots& \vdots &\cdots&\cdots&b_{r-2} &\vdots&\vdots&\cdots&\vdots\cr
0 &b_0&\vdots&\cdots&\cdots& 0&\vdots&\vdots&\cdots&\vdots\cr
0 & 0& b_0& \cdots&\cdots&\vdots&a_0 & \vdots&\cdots& \vdots\cr
\vdots & \vdots& \vdots& \vdots&\vdots&\vdots&0& a_0 &\cdots& \vdots\cr
\vdots & \vdots& \vdots& \vdots&\vdots&\vdots& \vdots & \vdots &\cdots& \vdots \cr
0 & 0 &0&\cdots&\cdots&b_0&0& 0 &\cdots& a_0 \cr
\end{pmatrix}
\end{split}
\end{equation}
}

This is the Sylvester matrix for $f$ and $g$. The rank of the Sylvester matrix for two polynomials when evaluated at the origin is the larger of the two degrees and thus the corank is the smaller. This gives the first entry of $TB(\mu_{n,r})$ for any $n \ge r$; $i_1=\corank (\delta I(\mu_{n,r}))=r$.

\medskip

Let $d_{n-1}=a_{n-1}-b_{r-1}$, $d_{n-j}=a_{n-j}-b_{r-j}-\underset{i+k=j}{\sum} d_{n-i} b_{r-k}, \hskip .1 cm \text{for} \hskip .1 cm j \le r$ and $d_{n-j}=a_{n-j}-\underset{i+k=j}{\sum} d_{n-i} b_{r-k}, \hskip .1 cm \text{for} \hskip .1 cm r<j \le n$. The following is true.

\begin{prop}
$\Delta^r I(\mu_{n,r})=I(\mu_{n,r})+(d_0, \cdots, d_{r-1})$
\end{prop}
\begin{proof}

By the definition of critical extension, $\Delta^r I(\mu_{n,r})$ is the sum of $I(\mu_{n,r})$ and the ideal spanned by all the $(n+1) \times (n+1)$ minors of $\delta I(\mu_{n,r})$. The later one is unchanged under elementary row operations on $\delta I(\mu_{n,r})$. We can do row operations on $\delta I(\mu_{n,r})$ as follows (in next section, we will describe these operations in matrix language). 

\medskip

Multiply the first row by $-b_{r-i}$ and add it to the $(i+1)$-th row for $i=1, \cdots, r$. After that, multiply the second row by $-b_{r-i}$ and add it to the $(i+2)$-th row for $i=1, \cdots, r$. Continue this process until all the $b_0, \cdots, b_{r-1}$ disappear from the first $n$ columns. After that, multiply the $(n+1)$-th row by $-b_{r-i}$ and add it to the $(n+i+1)$-th row for $i=1, \cdots, r-1$. For each $j=2, \cdots, r-1$, starting from $j=2$, we can multiply the $(n+j)$-th row by $-b_{r-i}$ and add it to the $(n+j+i)$-th row for $i=1, \cdots, r-j$. At the end, we get a matrix with the following form:

{\small
\begin{equation}
\begin{split}
\begin{pmatrix} 
1 & 0&0& \cdots&  0& 1& 0 & \cdots &0\cr
0 & 1 & 0   &\cdots&  0 & d_{n-1}&\ddots&\ddots &0\cr
0& 0 & 1 & \cdots& \cdots & d_{n-2}&\ddots&\ddots &\vdots\cr
\vdots & 0 &\cdots& \cdots&\cdots & \vdots&\ddots&\ddots &1\cr
\vdots &\vdots&\cdots&\cdots& 1 & \vdots&\ddots&\ddots &\vdots\cr
0 &\vdots& \cdots&\cdots& 0 &d_0&\ddots&\ddots &d_{r-1}\cr
\vdots & \vdots& \vdots&\vdots&\vdots & * &\ddots& \ddots & d_{r-2}\cr
\vdots & 0 &\cdots& \cdots&\cdots & *&*&\ddots &\vdots\cr
0 & 0 &\cdots&\cdots&0&* &*& * &d_0 \cr
\end{pmatrix}
\end{split}
\end{equation}
}

The elements in the position marked with \lq\lq *" in matrix $(2.2)$ can be generated by $d_0, \cdots, d_{r-1}$. Now it is easy to see that the ideal of all the $(n+1) \times (n+1)$ minors of the matrix $(2.2)$ is generated by $d_0, \cdots, d_{r-1}$, so is that of $\delta I(\mu_{n,r})$. Proposition $2.1$ follows.

\end{proof}

As an easy consequence of Proposition $2.1$, we have the following corollary.

\begin{cor}
$TB(I(\mu_{n,n}))=(n)=I(n,n)$ for any positive integer $n$.
\end{cor}

\begin{proof}

From the discussion on the first Jacobian, we know the first entry in $TB(I(\mu_{n,n}))$ is $n$.

\medskip

To show that $TB(I(\mu_{n,n}))=(n)$, we only need to prove that the corank of $\delta \Delta^n I(\mu_{n,n})$ evaluated at origin is $0$.

\medskip

By Proposition $2.1$, $\Delta^n I(\mu_{n,n})=I(\mu_{n,n})+(d_0, \cdots, d_{n-1})$, so $\delta \Delta^n I(\mu_{n,n})$ has the following form when evaluated at origin.
{\small
\begin{equation}
\begin{split}
\begin{pmatrix} 
I_n& I_n\cr
0 & 0   \cr
I_n& -I_n\cr
\end{pmatrix}
\end{split}
\end{equation}
}

whose corank is obviously equal to $0$, hence $TB(I(\mu_{n,n}))=(n)$.

\end{proof}

\section{Proof of Theorem 1.2}

\subsection{Toeplitz matrices}

Before we give a proof of Theorem $1.2$, let us discuss some properties on certain class of matrices called Toeplitz matrices.

A $n \times n$ Toeplitz matrix is a matrix in which each descending diagonal from left to right is constant. The lower shift matrix $L_n$ is a $n \times n$ binary matrix with ones only on the subdiagonal and zeroes elsewhere. It is obvious that $L_n$ is Toeplitz. Moreover, it is nilpotent.

A matrix $V$ is called a lower Toeplitz matrix if $V=vI_n+v_{m-1} L_n+v_{m-2} L_n^2+\cdots+v_{0} L_n^m$ for some $m (m \le n)$, where $I_n$ is the identity matrix and $v, v_0, \cdots, v_{m-1}$ are variables or constants.

The following lemma is true.

\begin{lem} Let $V=I_n+v_{m-1} L_n+v_{m-2} L_n^2+\cdots+v_{0} L_n^m$ and $W=I_n+w_{l-1} L_n+w_{l-2} L_n^2+\cdots+w_{0} L_n^l$ be two $n \times n$ lower Toeplitz matrices. Then 

\begin{enumerate}
\item $VW$ is a lower Toeplitz matrix and $VW=WV$.

\item $V^{-1}$ is a lower Toeplitz matrix and each entry below the diagonal is a polynomial in variables $v_0, \cdots, v_{m-1}$.

\end{enumerate}
\end{lem}

\begin{proof}

It is easy to see that $VW=(I_n+v_{m-1} L_n+v_{m-2} L_n^2+\cdots+v_{0} L_n^m) (I_n+w_{l-1} L_n+w_{l-2} L_n^2+\cdots+w_{0} L_n^l)=WV$ and $VW$ is a lower Toeplitz matrix.

Using long division to $\frac{1}{1+v_{m-1} L_n+v_{m-2} L_n^2+\cdots+v_{0} L_n^m}$ in the formal power series ring

\noindent $\frac{\C [v_0,\cdots,v_{m-1}] [[L_n]]}{(L_n^n)} $, we immediately have that $V^{-1}$ is a lower Toeplitz matrix and each entry below the diagonal is a polynomial in variables $v_0, \cdots, v_{m-1}$.
\end{proof}

\subsection {Proof of Theorem 1.2}

Let $A_{n+r+1}=I_{n+r+1}+a_{n-1} L_{n+r+1}+a_{n-2} L_{n+r+1}^2+\cdots+ a_{0} L_{n+r+1}^n$ and $B_{n+r+1}=I_{n+r+1}+b_{r-1} L_{n+r+1}+b_{r-2} L_{n+r+1}^2+\cdots+ b_{0} L_{n+r+1}^r$ , where $a_i, b_j$ are the coefficients of $f(x)$ and $g(x)$ respectively. The first Jacobian matrix is

{\small
\begin{equation}
\begin{split}
\delta I(\mu_{n,r})= \begin{pmatrix}\begin{pmatrix} I_{n+r}, 0 \cr \end{pmatrix} B_{n+r+1} \begin{pmatrix} I_n\cr 0 \cr \end{pmatrix}, \begin{pmatrix} I_{n+r}, 0 \cr \end{pmatrix} A_{n+r+1} \begin{pmatrix} I_r\cr 0 \cr \end{pmatrix}\end{pmatrix}\end{split}
\end{equation}
}
 
Let $D_{n+r+1}=(B_{n+r+1})^{-1} A_{n+r+1}$. It is easy to see that the row operations we did in section $2$ on $\delta I(\mu_{n,r})$ is exactly multiplying $\delta I(\mu_{n,r})$ by $(B_{n+r})^{-1}$ on the left, where $B_{n+r}=\begin{pmatrix} I_{n+r}, 0 \cr \end{pmatrix} B_{n+r+1} \begin{pmatrix} I_{n+r}\cr 0 \cr \end{pmatrix}$. So  
$D_{n+r+1}=I_{n+r+1}+d_{n-1} L_{n+r+1}+d_{n-2} L_{n+r+1}^2+\cdots+ d_{0} L_{n+r+1}^n+d_{-1} L_{n+r+1}^{n+1}+\cdots+d_{-r} L_{n+r+1}^{n+r}$ for some $d_{-1}, \cdots, d_{-r}$. Comparing the corresponding coefficients $L_{n+r+1}^{n+j}$ for $j=1, \cdots, r$ on both sides of the equation $B_{n+r+1} D_{n+r+1} =A_{n+r+1}$, we have that $d_{-j}+b_{r-1} d_{-j+1}+\cdots+b_0 d_{r-j}=0$, which imply that $d_{-j}, j=1, \cdots, r$ are generated by $d_0, \cdots, d_{r-1}$. This coincides with what we said about the elements in the position marked with \lq\lq *" in matrix $(2.2)$.

From $A_{n+r+1}=B_{n+r+1} D_{n+r+1}$, it is easy to get the following equations:

{\small
$$\left\{\begin{array}{l} 
       a_{n-1}=d_{n-1} + b_{r-1} \\ 
       a_{n-2}=d_{n-2} + b_{r-2}+ b_{r-1} d_{n-1}  \\ 
       and\\
       a_{n-j}=d_{n-j} +b_{r-j} + \underset{i+k=j}{\sum} b_{r-k} d_{n-i} , \hskip .1 cm \text{for} \hskip .1 cm j \le r \\
       a_{n-j}=d_{n-j} + \underset{i+k=j}{\sum} b_{r-k} d_{n-i} , \hskip .1 cm \text{for} \hskip .1 cm r<j \le n \\
     \end{array} \right. $$
}
Taking derivatives with respect to $a_i$'s and $b_i$'s in descending order from $n-1$ to $0$ and $r-1$ to $0$ respectively in the above equations and using the Chain Rule, we have that

{\small
\begin{equation}
\begin{split}
\begin{pmatrix} I_n, & 0\cr \end{pmatrix} \end{split} = \begin{split}
\begin{pmatrix} B, & D\cr \end{pmatrix} \end{split}
\begin{split}
\begin{pmatrix} 
\begin{pmatrix} \frac{\partial d_i}{\partial a_j} \end{pmatrix}& \begin{pmatrix} \frac{\partial d_i}{\partial b_j} \end{pmatrix}\cr
0 & I_r\cr
\end{pmatrix}
\end{split}
\end{equation}
}

\noindent where {\small $B= (I_n, 0) B_{n+r+1} \begin{pmatrix} I_n\cr 0 \cr \end{pmatrix}$} and {\small $D= (I_n, 0) D_{n+r+1} \begin{pmatrix} I_r\cr 0 \cr \end{pmatrix}$}. This gives that

{\small
\begin{equation}
\begin{split}
B \begin{pmatrix} \frac{\partial d_i}{\partial a_j} \end{pmatrix}=I_n\\
B \begin{pmatrix} \frac{\partial d_i}{\partial b_j} \end{pmatrix} +D=0\\
\end{split}
\end{equation}
}

Let {\small $A= (I_n, 0) A_{n+r+1} \begin{pmatrix} I_r\cr 0 \cr \end{pmatrix}$}, then {\small $A= (I_n, 0) A_{n+r+1} \begin{pmatrix} I_r\cr 0 \cr \end{pmatrix}=(I_n, 0)B_{n+r+1}D_{n+r+1}\begin{pmatrix} I_r\cr 0 \cr \end{pmatrix}=(B, 0) \begin{pmatrix} D\cr * \cr \end{pmatrix}=BD$}.

Using Equation $(3.3)$ and {\small $A=BD$}, we can prove the following lemma.

\begin{lem} 
{\small $\begin{pmatrix} \frac{\partial^s d_{i}}{\partial^{s-1} b_0 \partial b_j} \end{pmatrix}=-s\begin{pmatrix} \frac{\partial^s d_{i}}{\partial^{s-1} b_0 \partial a_j} \end{pmatrix} D$} for $s=1, \cdots, q_1$.
\end{lem}

\begin{proof} It is easy to see that $B= (I_n, 0) B_{n+r+1} \begin{pmatrix} I_n\cr 0 \cr \end{pmatrix}=I_n+ b_{r-1} L_n+ \cdots+ b_{0} L_n^r$ is a lower Toeplitz matrix. Its derivative with respect to $b_0$ is again a lower Toeplitz matrix, in fact, $\frac{\partial B}{\partial b_0}=L_n^r$. So $B\frac{\partial B}{\partial b_0}=\frac{\partial B}{\partial b_0} B$ by Lemma $3.1$ or direct verification.

The equation $\frac{\partial B}{\partial b_0}=L_n^r$ implies that any higher derivatives of $B$ with respect to $b_0$ is zero. From 
$B \begin{pmatrix} \frac{\partial d_i}{\partial a_j} \end{pmatrix}=I_n$ in Equation $(3.3)$, we have

\begin{equation}
\begin{split}
\begin{pmatrix} \frac{\partial d_i}{\partial a_j} \end{pmatrix}=B^{-1}
\end{split}
\end{equation}

Taking derivatives with respect to $b_0$ repeatedly on both sides of Equation $(3.4)$ gives

\begin{equation}
\begin{split}
\begin{pmatrix} \frac{\partial^s d_{i}}{\partial^{s-1} b_0 \partial a_j} \end{pmatrix}=(-1)(-2) \cdots (-s+1) B^{-s} \left(\frac{\partial B}{\partial b_0} \right)^{s-1}=(-1)^{s-1} (s-1)! B^{-s} \left(\frac{\partial B}{\partial b_0} \right)^{s-1}
\end{split}
\end{equation}

\noindent for any positive integer $s$.

From $B \begin{pmatrix} \frac{\partial d_i}{\partial b_j} \end{pmatrix} +D=0$ and $A=BD$, we have

\begin{equation}
\begin{split}
\begin{pmatrix} \frac{\partial d_i}{\partial b_j} \end{pmatrix}=-B^{-1}D=-B^{-2} A
\end{split}
\end{equation}

Taking derivatives with respect to $b_0$ repeatedly on both sides of Equation $(3.6)$ and using the commutativity $B\frac{\partial B}{\partial b_0}=\frac{\partial B}{\partial b_0} B$ give that

\begin{equation}
\begin{split}
\begin{pmatrix} \frac{\partial^s d_{i}}{\partial^{s-1} b_0 \partial b_j} \end{pmatrix}=(-1)(-2) \cdots (-s) B^{-s-1} \left(\frac{\partial B}{\partial b_0} \right)^{s-1} A=(-1)^{s} s! B^{-s} \left(\frac{\partial B}{\partial b_0} \right)^{s-1} D
\end{split}
\end{equation}

Now our lemma follows immediately from Equations $(3.5)$ and $(3.7)$.

\end{proof}

We also need the following lemma.

\begin{lem}
The $n \times r$ matrix $\begin{pmatrix} \frac{\partial d_i}{\partial b_j} \end{pmatrix}$ in Equation $(3.3)$ is the first $r$ columns in a $n \times n$ lower Toeplitz matrix; moreover, the elements $\frac{\partial d_i}{\partial b_j}$ for $i=0, \cdots, r-1$ and $j=0, \cdots, r-1$ can be generated by $\frac{\partial d_{r-1}}{\partial b_j}, j=0, \cdots, r-1$ and $d_1, \cdots, d_{r-1}$ if $q_1 \ge 2$.
\end{lem}
\begin{proof}
From $B \begin{pmatrix} \frac{\partial d_i}{\partial b_j} \end{pmatrix} +D=0$ in Equation $(3.3)$, we have that 
$\begin{pmatrix} \frac{\partial d_i}{\partial b_j} \end{pmatrix}=-B^{-1} D$. Let $\widehat{D}=I_n+d_{n-1} L_n+\cdots+d_1 L_n^{n-1}$, which is a lower Toeplitz matrix. Then $D=\widehat{D} \begin{pmatrix} I_r\\ 0 \end{pmatrix}$. By Lemma $3.1$, $-B^{-1}\widehat{D}$ is a lower Toeplitz matrix. So $\begin{pmatrix} \frac{\partial d_i}{\partial b_j} \end{pmatrix}=-B^{-1} D=-B^{-1} \widehat{D} \begin{pmatrix} I_r\\ 0 \end{pmatrix}$ is the first $r$ columns in the $n \times n$ lower Toeplitz matrix $-B^{-1} \widehat{D}$.

Denote $B^{-1} \widehat{D}$ as $I_n+t_{n-1}L_n+\cdots+t_1 L_n^{n-1}$, we have that

\begin{equation}
\begin{split}
I_n+d_{n-1} L_n+\cdots+d_1 L_n^{n-1}=\widehat{D}=B B^{-1} \widehat{D}=\cr(I_n+b_{r-1}L_n+\cdots+b_0 L_n^{r})(I_n+t_{n-1}L_n+\cdots+t_1 L_n^{n-1})
\end{split}
\end{equation}

Comparing the coefficients of $L_n^{k}$ for $k=n-r+1, \cdots, n-1$ in Equation $(3.8)$, we have that 

\begin{equation}
\begin{split}
d_i=t_i+b_{r-1}t_{i+1}+ \cdots +b_0t_{r+i} \hskip .2 cm \text{for} \hskip .2 cm i=1, \cdots, r-1
\end{split}
\end{equation}

From Equation $(3.9)$, it is easy to see that $t_1, \cdots, t_{r-1}$ are generated by $t_r, \cdots, t_{2r-1}$ and $d_1, \cdots, d_{r-1}$. From the equation $\begin{pmatrix} \frac{\partial d_i}{\partial b_j} \end{pmatrix}=-B^{-1} \widehat{D} \begin{pmatrix} I_r, & 0 \end{pmatrix}$ we see that $t_k=-\frac{\partial d_i}{\partial b_j}$ for $k=r-j+i$ , where $i=0, \cdots, r-1$ and $j=0, \cdots, r-1$. Hence the elements $\frac{\partial d_i}{\partial b_j}$ for $i=0, \cdots, r-1$ and $j=0, \cdots, r-1$ can be generated by $t_{2r-1-j}=\frac{\partial d_{r-1}}{\partial b_j}, j=0, \cdots, r-1$ and $d_1, \cdots, d_{r-1}$.

\end{proof}

As an easy consequence of Lemma $3.2$ and Lemma $3.3$, we have

\begin{prop}
$\Delta^r (\Delta^r I(\mu_{n,r}))=I(\mu_{n,r})+(d_0, \cdots, d_{r-1}, \frac{\partial d_{r-1}}{\partial b_{r-1}}, \cdots, \frac{\partial d_{r-1}}{\partial b_0})$ if $q_1 \ge 2$.
\end{prop}

\begin{proof}

By Proposition $2.1$, we have that $\Delta^r I(\mu_{n,r})=I(\mu_{n,r})+(d_0, \cdots, d_{r-1})$. To prove this corollary, we only need to show that the corank of $\delta (I(\mu_{n,r})+(d_0, \cdots, d_{r-1}))$ evaluated at origin is $r$ and

\begin{equation}
\begin{split}
\Delta^r (I(\mu_{n,r})+(d_0, \cdots, d_{r-1}))= I(\mu_{n,r})+(d_0, \cdots, d_{r-1}, \frac{\partial d_{r-1}}{\partial b_{r-1}}, \cdots, \frac{\partial d_{r-1}}{\partial b_0})
\end{split}
\end{equation}

Because $\delta (I(\mu_{n,r})+(d_0, \cdots, d_{r-1}))=\begin{pmatrix} \delta (I(\mu_{n,r})) \\ \left(\frac{\partial d_i}{\partial a_j}\right), \left(\frac{\partial d_i}{\partial b_j}\right) \end{pmatrix}= \begin{pmatrix}  B_{n+r} \begin{pmatrix} I_n \\ 0\cr \end{pmatrix} &A_{n+r} \begin{pmatrix} I_r \\ 0\cr \end{pmatrix}\\ \left(\frac{\partial d_i}{\partial a_j}\right)& \left(\frac{\partial d_i}{\partial b_j}\right) \end{pmatrix}$, left multiplying $\delta (I(\mu_{n,r})+(d_0, \cdots, d_{r-1}))$ by $\begin{pmatrix} I_n &0&0 \\ 0&I_r&0\\ -\left(\frac{\partial d_i}{\partial a_j}\right)&0&I_r \end{pmatrix} \begin{pmatrix} (B_{n+r})^{-1} &0 \\ 0&I_r\end{pmatrix}$ gives 

\begin{equation}
\begin{split}
\begin{pmatrix} I_n &D \\ 0&*\\ 0&-\left(\frac{\partial d_i}{\partial a_j}\right)D+ \left(\frac{\partial d_i}{\partial b_j}\right) \end{pmatrix}
\end{split}
\end{equation}

\noindent where elements in the position marked by \lq\lq *" can be generated by $d_0, \cdots, d_{r-1}$.

By Lemma $3.2$ (the case $s=1$), we can rewrite the above matrix as 

\begin{equation}
\begin{split}
\begin{pmatrix} I_n &D \\ 0&*\\ 0&2\left(\frac{\partial d_i}{\partial b_j}\right) \end{pmatrix}
\end{split}
\end{equation}

The $*$ part is given by $\begin{pmatrix} 0, I_r\end{pmatrix} (B_{n+r})^{-1} A_{n+r} \begin{pmatrix} I_r\\0\end{pmatrix}$, which evaluated at origin is $\begin{pmatrix} 0, I_r\end{pmatrix} $

\noindent $\begin{pmatrix} I_r\\0\end{pmatrix}=0$ because $q_1 \ge 2$. The same argument gives that $\left(\frac{\partial d_i}{\partial b_j}\right)=-\begin{pmatrix} 0, I_r\end{pmatrix} B^{-1}\widehat{D} \begin{pmatrix} I_r\\ 0 \end{pmatrix}$ evaluated at origin is also equal to $\begin{pmatrix} 0, I_r\end{pmatrix} \begin{pmatrix} I_r\\0\end{pmatrix}=0$. Therefor the corank of $\delta (I(\mu_{n,r})+(d_0, \cdots, d_{r-1}))$ evaluated at origin is $r$. 

By Lemma $3.3$, any element in the $r \times r$ matrix $\left(\frac{\partial d_i}{\partial b_j}\right)$ can be generated by $\frac{\partial d_{r-1}}{\partial b_j}, j=0, \cdots, r-1$ and $d_1, \cdots, d_{r-1}$. So any $(n+1) \times (n+1)$ minor of the matrix in $(3.12)$ can be generated by $\frac{\partial d_{r-1}}{\partial b_j}, j=0, \cdots, r-1$ and $d_0, \cdots, d_{r-1}$ because it must have at least one row whose elements are from $*$ or $2 \left(\frac{\partial d_i}{\partial b_j}\right)$. This implies that 

\begin{equation}
\begin{split}
\Delta^r (I(\mu_{n,r})+(d_0, \cdots, d_{r-1})) \subseteq I(\mu_{n,r})+(d_0, \cdots, d_{r-1}, \frac{\partial d_{r-1}}{\partial b_{r-1}}, \cdots, \frac{\partial d_{r-1}}{\partial b_0})
\end{split}
\end{equation}

For each $j=0, \cdots, r-1$, the $(n+1) \times (n+1)$ minor $\begin{pmatrix} I_n &\# \\ 0&2 \frac{\partial d_{r-1}}{\partial b_j} \end{pmatrix}$ has determinant $2 \frac{\partial d_{r-1}}{\partial b_j}$, so the $\subseteq$ in Equation $(3.13)$ is actually an equality. This proves Proposition $3.4$.

\end{proof}

Let $\psi_i=d_i$ for $i=0, \cdots, r-1$ and $\psi_{sr+i}=\frac{\partial^s d_{r-1}}{\partial^{s-1} b_0 \partial b_{r-1-i}}$ for $i=0, \cdots, r-1$ and $s=1, \cdots, q_1-1$. We have the following lemma.

\begin{lem}
$\begin{pmatrix} \frac{\partial \psi_{sr+i}}{\partial b_j}\end{pmatrix}=- (s+1) \begin{pmatrix} \frac{\partial \psi_{sr+i}}{\partial a_j} \end{pmatrix} D$ for $s=0, \cdots, q_1-1$. 
\end{lem}

\begin{proof}

Equation $(3.3)$ implies that $\begin{pmatrix} \frac{\partial d_i}{\partial a_j} \end{pmatrix}=B^{-1}$ and  
$\begin{pmatrix} \frac{\partial d_i}{\partial b_j} \end{pmatrix}=-B^{-1}D$. So $\begin{pmatrix} \frac{\partial d_i}{\partial b_j} \end{pmatrix}=-\begin{pmatrix} \frac{\partial d_i}{\partial a_j} \end{pmatrix} D$. Lemma $3.5$ is true in the case $s=0$.

For $s \ge 1$, by Lemma $3.3$ we have that $\frac{\partial d_{r-1}}{\partial b_{r-1-i}}=\frac{\partial d_{i}}{\partial b_0}$, so $\psi_{sr+i}=\frac{\partial^s d_{r-1}}{\partial^{s-1} b_0 \partial b_{r-1-i}}=\frac{\partial^{s-1} }{\partial^{s-1} b_0} \left(\frac{\partial d_{r-1}}{\partial b_{r-1-i}}\right)=\frac{\partial^{s-1} }{\partial^{s-1} b_0} \left(\frac{\partial d_{i}}{\partial b_0}\right)=\frac{\partial^s d_{i}}{\partial^{s} b_0}$.

By Lemma $3.2$ we have that {\small $\begin{pmatrix} \frac{\partial \psi_{sr+i}}{\partial b_j}\end{pmatrix}=\begin{pmatrix} \frac{\partial}{\partial b_j} \frac{\partial^s d_{i}}{\partial^{s} b_0}\end{pmatrix}= \frac{\partial}{\partial b_0} \begin{pmatrix} \frac{\partial^s d_{i}}{\partial^{s-1} b_0 \partial b_j} \end{pmatrix}=\frac{\partial}{\partial b_0} \left(-s \frac{\partial^s d_{i}}{\partial^{s-1} b_0\partial a_j} D\right)$}

\noindent $=-s \begin{pmatrix} \frac{\partial}{\partial a_j}\frac{\partial^s d_{i}}{\partial^{s} b_0} \end{pmatrix} D-s \left(\frac{\partial^s d_{i}}{\partial^{s-1} b_0\partial a_j}\right) \frac{\partial D}{\partial b_0}$. Our lemma follows if we can show that $-s \left(\frac{\partial^s d_{i}}{\partial^{s-1} b_0\partial a_j}\right) $

\noindent $\frac{\partial D}{\partial b_0}=-\begin{pmatrix} \frac{\partial}{\partial a_j}\frac{\partial^s d_{i}}{\partial^{s} b_0} \end{pmatrix} D$. This can be done as follows.

From the equation $A=BD$, we have that $\frac{\partial B}{\partial b_0} D+B\frac{\partial D}{\partial b_0}=0$. So $\frac{\partial D}{\partial b_0}=-B^{-1} \frac{\partial B}{\partial b_0} D$. Applying Equation $(3.5)$ to both indices $s$ and $s+1$, we have that $-s \left(\frac{\partial^s d_{i}}{\partial^{s-1} b_0\partial a_j}\right) \frac{\partial D}{\partial b_0}=s \left(\frac{\partial^s d_{i}}{\partial^{s-1} b_0\partial a_j}\right) B^{-1} \frac{\partial B}{\partial b_0} D=(-1)^{s-1} s! B^{-(s+1)} \left(\frac{\partial B}{\partial b_0}\right)^s D=-\begin{pmatrix} \frac{\partial}{\partial a_j}\frac{\partial^s d_{i}}{\partial^{s} b_0} \end{pmatrix} D$. This completes the proof of Lemma $3.5$.

\end{proof}

\begin{thm}
$\overset{s}{\overbrace{\Delta^r \cdots \Delta^r}} I(\mu_{n,r})= I(\mu_{n,r})+(\psi_0, \cdots, \psi_{sr-1})$ for $s=1, \cdots, q_1$.
\end{thm}

\begin{proof}
The case $s=1$ has been proved in Proposition $2.1$. If $q_1=1$, we are done. So we may assume that $q_1 \ge 2$.

Suppose that Theorem $3.6$ is true for $s=1, \cdots, p$. By Proposition $2.1$ and Proposition $3.4$, we may assume $p \ge 2$. If $p=q_1$, we are done. Otherwise we may assume that $p \le q_1-1$.

By the inductive assumption, we have that

\begin{equation}
\begin{split}
\overset{p}{\overbrace{\Delta^r \cdots \Delta^r}} I(\mu_{n,r})= I(\mu_{n,r})+(\psi_0, \cdots, \psi_{pr-1})
\end{split}
\end{equation}

We need to prove that the corank of $\delta (I(\mu_{n,r})+(\psi_0, \cdots, \psi_{pr-1}))$ evaluated at origin is $r$ and

\begin{equation}
\begin{split}
\Delta^r (\overset{p}{\overbrace{\Delta^r \cdots \Delta^r}} I(\mu_{n,r}))= \Delta^r(I(\mu_{n,r})+(\psi_0, \cdots, \psi_{pr-1}))=I(\mu_{n,r})+(\psi_0, \cdots, \psi_{(p+1)r-1})
\end{split}
\end{equation}

It is easy to see that $\delta (I(\mu_{n,r})+(\psi_0, \cdots, \psi_{pr-1}))=\begin{pmatrix} \delta (I(\mu_{n,r})) \\ \left(\frac{\partial \psi_{sr+i}}{\partial a_j}, \frac{\partial \psi_{sr+i}}{\partial b_j}\right) \end{pmatrix}$, where $s$ varies from $0$ to $p-1$ and $i$ from $0$ to $r-1$ respectively. Left multiplying $\delta (I(\mu_{n,r})+(\psi_0, \cdots, \psi_{pr-1}))$ by $\begin{pmatrix} \begin{pmatrix} I_n &0 \\0&I_r \end{pmatrix} &0&\cdots&0 \\ -\left(\frac{\partial \psi_{i}}{\partial a_j}\right)&I_r&\cdots&0\\\vdots&\vdots&\ddots &\vdots\\-\frac{\partial \psi_{(p-1)r+i}}{\partial a_j}&0&\cdots&I_r\end{pmatrix} \begin{pmatrix} (B_{n+r})^{-1} &0&\cdots&0 \\ 0&I_r&\cdots&\vdots\\\vdots&\ddots&\ddots &0\\0&\cdots&0&I_r\end{pmatrix}$, we get the following matrix

$$\begin{pmatrix} I_n &D \\0&*\\ 0&-\left(\frac{\partial \psi_{sr+i}}{\partial a_j}\right)D+ \left(\frac{\partial \psi_{sr+i}}{\partial b_j}\right) \end{pmatrix}$$

\noindent where elements in the position marked by \lq\lq *" can be generated by $d_0, \cdots, d_{r-1}$ and $s$ varies from $0$ to $p-1$.

By Lemma $3.5$ and induction assumption, it is equal to 

\begin{equation}
\begin{split}
\begin{pmatrix} I_n &D \\ 0&*\\ 0&(\frac{1}{p}+1)\left(\frac{\partial \psi_{(p-1)r+i}}{\partial b_j}\right) \end{pmatrix}
\end{split}
\end{equation}

\noindent where elements in the position marked by \lq\lq *" can be generated by $\psi_0, \cdots, \psi_{pr-1}$.

By induction assumption, the corank of $\begin{pmatrix} I_n &D \\ 0&*\end{pmatrix}$ evaluated at origin is $r$. To show that the corank of $\delta (I(\mu_{n,r})+(\psi_0, \cdots, \psi_{pr-1}))$ evaluated at origin is $r$, it is sufficient to prove $\left(\frac{\partial \psi_{(p-1)r+i}}{\partial b_j}\right)$ is zero when evaluated at origin.

The matrix {\small $\left(\frac{\partial \psi_{(p-1)r+i}}{\partial b_j}\right)=\left(\frac{\partial }{\partial b_j} (\frac{\partial^{p-1} d_{r-1}}{\partial^{p-2} b_0 \partial b_{r-1-i}}) \right)=\left(\frac{\partial }{\partial b_j} (\frac{\partial^{p-2}}{ \partial^{p-2} b_0} \frac{\partial d_{r-1}}{\partial b_{r-1-i}}) \right)=\left( \frac{\partial }{\partial b_j} (\frac{\partial^{p-2}}{ \partial^{p-2} b_0} \frac{\partial d_{i}}{\partial b_{0}})\right)$

\noindent $=\left(\frac{\partial^{p-1}}{ \partial^{p-1} b_0}(\frac{\partial d_i}{\partial b_j}) \right)$}. By Equation $(3.6)$, we have that $\left(\frac{\partial^{p-1}}{ \partial^{p-1} b_0}(\frac{\partial d_i}{\partial b_j}) \right)=-\begin{pmatrix} 0,I_r\end{pmatrix}\left(\frac{\partial^{p-1}}{ \partial^{p-1} b_0} (B^{-2}) \right)A $

\noindent $=(-1)^p p! \begin{pmatrix} 0,I_r\end{pmatrix} B^{-p-1} (\frac{\partial B}{\partial b_0})^{p-1} A=(-1)^p p! \begin{pmatrix} 0,I_r\end{pmatrix} B^{-p-1} L_n^{(p-1)r} A$, which evaluated at origin is zero because $B^{-p-1} L_n^{(p-1)r} A$ evaluated at origin has the form $\begin{pmatrix} 0\\I_r\\0\end{pmatrix}$ and the bottom $0$ consists of $(q_1-p)r \ge r$ rows.

By the definition $\frac{\partial \psi_{(p-1)r+i}}{\partial b_j}=\frac{\partial }{\partial b_j} (\frac{\partial^{p-1} d_{r-1}}{\partial^{p-2} b_0 \partial b_{r-1-i}}) =\frac{\partial }{\partial b_j} (\frac{\partial^{p-2}}{ \partial^{p-2} b_0} \frac{\partial d_{r-1}}{\partial b_{r-1-i}}) =\frac{\partial }{\partial b_j} (\frac{\partial^{p-2}}{ \partial^{p-2} b_0} \frac{\partial d_{i}}{\partial b_{0}})=\frac{\partial}{\partial b_0}\frac{\partial^{p-2}}{ \partial^{p-2} b_0}(\frac{\partial d_i}{\partial b_j})=\frac{\partial}{\partial b_0} \frac{\partial \psi_{(p-2)r+i}}{\partial b_j}$. By induction assumption $\frac{\partial \psi_{(p-2)r+i}}{\partial b_j}$ can be generated by $\psi_0, \cdots, \psi_{(p-1)r+r-1}$ and $\frac{\partial \psi_{(p-1)r+i}}{\partial b_0} $ for $i=0, \cdots, r-1$. By the definition $\frac{\partial \psi_{(p-1)r+i}}{\partial b_0} =\frac{\partial}{\partial b_0} (\frac{\partial^{p-1} d_{r-1}}{\partial^{p-2} b_0 \partial b_{r-1-i}})=\frac{\partial^{p} d_{r-1}}{\partial^{p-1} b_0 \partial b_{r-1-i}}=\psi_{pr+i}$. Therefore any $(n+1) \times (n+1)$ minor of matrix $(3.16)$ and hence $\delta (I(\mu_{n,r})+(\psi_0, \cdots, \psi_{pr-1}))$ can be generated by $(\psi_0, \cdots, \psi_{(p+1)r-1})$, this implies that $\Delta^r(I(\mu_{n,r})+(\psi_0, \cdots, \psi_{pr-1})) \subseteq I(\mu_{n,r})+(\psi_0, \cdots, \psi_{(p+1)r-1})$. Actually the inequality is an equality because each $\psi_{pr+i}$ is only different from a $(n+1) \times (n+1)$ minor of $\delta (I(\mu_{n,r})+(\psi_0, \cdots, \psi_{pr-1}))$ by a nonzero constant. Theorem $3.6$ follows.
\end{proof}

As an easy corollary of Theorem $3.6$ and its proof, we have

\begin{cor}
The first $q_1$ entries in $TB(I(\mu_{n,r}))$ are $(r, \cdots, r)$.
\end{cor}

Our next goal is to prove that

\begin{prop}
The $(q_1+1)$-th entry in $TB(I(\mu_{n,r}))$ is $r_1$.
\end{prop}
\begin{proof}
By Theorem $3.6$, it is sufficient to prove that the rank of $\delta(I(\mu_{n,r})+(\psi_0, \cdots, \psi_{q_1r-1}))$ evaluated at origin is $n+r-r_1$. By the proof of Theorem $3.6$, we only need to prove that the rank of $\begin{pmatrix}\frac{\partial \psi_{(q_1-1)r+i}}{\partial b_j}\end{pmatrix}$ evaluated at origin is $r-r_1$, where $i=0, \cdots, r-1$ and $j=0, \cdots, r-1$.

When $q_1=1$, we have that $\begin{pmatrix}\frac{\partial \psi_{(q_1-1)r+i}}{\partial b_j}\end{pmatrix}= \begin{pmatrix}\frac{\partial d_i}{\partial b_j}\end{pmatrix}$. By Equation $(3.6)$ and that $B^{-1}$ is equal to $I_n$ when evaluated at origin, we have that $\begin{pmatrix}\frac{\partial d_i}{\partial b_j}\end{pmatrix}$ evaluated at origin has the same rank as that of $-\begin{pmatrix}0&I_r\end{pmatrix} \begin{pmatrix} I_r\\0 \end{pmatrix}$, the latter one has rank $r-r_1$ because $n=r_1+r$ and $\begin{pmatrix}0&I_r\end{pmatrix} \begin{pmatrix} I_r\\0 \end{pmatrix}$ represents the first $r$ columns in the $r \times n$ matrix $\begin{pmatrix}0&I_r\end{pmatrix}$.

When $q_1>1$, we have that $\begin{pmatrix}\frac{\partial \psi_{(q_1-1)r+i}}{\partial b_j}\end{pmatrix}=\begin{pmatrix} \frac{\partial}{\partial b_j} \frac{\partial^{q_1-1} d_{r-1}}{\partial^{q_1-2} b_0 \partial b_{r-1-i}}\end{pmatrix}= \begin{pmatrix} \frac{\partial}{\partial b_j} \frac{\partial^{q_1-2} }{\partial^{q_1-2} b_0 } \frac{\partial d_{r-1}}{\partial b_{r-1-i}}\end{pmatrix}$

\noindent $=\begin{pmatrix} \frac{\partial}{\partial b_j} \frac{\partial^{q_1-2} }{\partial^{q_1-2} b_0 } \frac{\partial d_{i}}{\partial b_{0}}\end{pmatrix}=\begin{pmatrix}\frac{\partial^{q_1} d_i}{\partial^{q_1-1} b_0 \partial b_j}\end{pmatrix}$. By Equation $(3.7)$, we have that $\begin{pmatrix}\frac{\partial^{q_1} d_i}{\partial^{q_1-1} b_0 \partial b_j}\end{pmatrix}= \begin{pmatrix}0&I_r\end{pmatrix}  $

\noindent $(-1)^{q_1} q_1! B^{-q_1} \left(\frac{\partial B}{\partial b_0} \right)^{q_1-1} D$, which evaluated at origin has the same rank as $\begin{pmatrix}0&0&I_r\end{pmatrix}$

\noindent $L_n^{(q_1-1)r}\begin{pmatrix} I_r\\0\\0\end{pmatrix}$, where the first zero in the $r \times n$ matrix $\begin{pmatrix}0&0&I_r\end{pmatrix}$ represents the first $(q_1-1)r$ columns and the second zero represents the next $r_1$ columns. Because $L_n^{(q_1-1)r}= \begin{pmatrix} 0&0&0\\I_{r_1}&0&0\\0&I_r&0\end{pmatrix}$, so $\begin{pmatrix}0&0&I_r\end{pmatrix} L_n^{(q_1-1)r}= \begin{pmatrix}0&I_r&0\end{pmatrix}$, where the first zero in $\begin{pmatrix}0&I_r&0\end{pmatrix}$ occupies the first $r_1$ columns and the second one occupies the last $(q_1-1)r$ columns. It is easy to see that $\begin{pmatrix}0&0&I_r\end{pmatrix} L_n^{(q_1-1)r} \begin{pmatrix} I_r\\0\\0\end{pmatrix}=\begin{pmatrix}0&I_r&0\end{pmatrix} \begin{pmatrix} I_r\\0\\0\end{pmatrix}$ has rank $r-r_1$.
\end{proof}

In order to obtain the $(q_1+1)$-th critical extension $\Delta^{r_1}(\overset{q_1}{\overbrace{\Delta^r \cdots \Delta^r}} I(\mu_{n,r}))$, we need a key lemma.

Denote $B^{-q_1}\widehat{D}=I_n+\alpha_{n-1} L_n+\cdots+\alpha_1 L_n^{n-1}$. We have that

$\begin{pmatrix}0&I_r\end{pmatrix} L_n^{(q_1-1)r} B^{-q_1}\widehat{D} \begin{pmatrix} I_r\\ 0 \end{pmatrix}=\begin{pmatrix} \alpha_{q_1r} &\cdots & \alpha_{n-1}& 1& \cdots&0\\\vdots&\cdots&\cdots&\alpha_{n-1} &\ddots& \vdots \\ \vdots&\cdots&\cdots&\vdots & \cdots & 1 \\ \vdots & \vdots & \vdots \\ \alpha_{(q_1-1)r+1}&\cdots&\alpha_{n-r}&\alpha_{n-r+1} & \cdots & \alpha_{q_1r}\\\end{pmatrix}$

\medskip

It is easy to see that

$\begin{pmatrix}0&I_r\end{pmatrix} L_n^{(q_1-1)r} B^{-q_1}\widehat{D} \begin{pmatrix} I_r\\ 0 \end{pmatrix} \begin{pmatrix} 0\\ I_{r-r_1} \end{pmatrix}=(I_r+\alpha_{n-1}L_r+\cdots+\alpha_{n-r+1} L_r^{r-1})\begin{pmatrix} I_{r-r_1}\\ 0 \end{pmatrix}$ 

\medskip

The matrix $(I_r+\alpha_{n-1}L_r+\cdots+\alpha_{n-r+1} L_r^{r-1})^{-1} \begin{pmatrix} 0, I_r \end{pmatrix} L_n^{(q_1-1)r} B^{-q_1}\widehat{D} \begin{pmatrix} I_r\\ 0 \end{pmatrix}$ has the form

$$\begin{pmatrix} * &I_{r-r_1}\\ K & 0 \end{pmatrix}$$

\noindent where $K=\begin{pmatrix} 0, I_{r_1} \end{pmatrix} (I_r+\alpha_{n-1}L_r+\cdots+\alpha_{n-r+1} L_r^{r-1})^{-1} \begin{pmatrix} 0, I_r \end{pmatrix} L_n^{(q_1-1)r} B^{-q_1}\widehat{D} \begin{pmatrix} I_r\\ 0 \end{pmatrix} \begin{pmatrix} I_{r_1}\\ 0 \end{pmatrix}$ is a $r_1 \times r_1$ matrix. 

We have the following lemma.

\begin{lem} 
The elements in the first row of $K$ and $\psi_0, \cdots, \psi_{q_1r-1}$ generate all elements in $K$.
\end{lem}

\begin{proof}

Denote $\Phi=I_r+\alpha_{n-1}L_r+\cdots+\alpha_{n-r+1} L_r^{r-1}$ and $(I_n+\alpha_{n-1}L_n+\cdots+\alpha_{1} L_n^{n-1})^{-1}=I_n+\beta_{n-1}L_n+\cdots+\beta_{1} L_n^{n-1}$. Using the partition $(r_1, r, (q_1-1)r)$ of $n$, we can split $I_n+\beta_{n-1}L_n+\cdots+\beta_{1} L_n^{n-1}$ and $I_n+\alpha_{n-1}L_n+\cdots+\alpha_{1} L_n^{n-1}$ into $3 \times 3$ block matrices. Comparing the $(2,1)$ block in $(I_n+\beta_{n-1}L_n+\cdots+\beta_{1} L_n^{n-1})(I_n+\alpha_{n-1}L_n+\cdots+\alpha_{1} L_n^{n-1})=I_n$ we have that

$\begin{pmatrix} \beta_{q_1r} &\cdots& \beta_{n-1} \\ \vdots & \vdots & \vdots \\ \beta_{(q_1-1)r+1} &\cdots &\beta_{n-r} \end{pmatrix} \begin{pmatrix} 1 &\cdots& 0 \\ \vdots & \ddots & \vdots \\ \alpha_{n-r_1+1} &\cdots &1 \end{pmatrix}+\Phi^{-1}\begin{pmatrix} \alpha_{q_1r} &\cdots& \alpha_{n-1} \\ \vdots & \vdots & \vdots \\ \alpha_{(q_1-1)r+1} &\cdots &\alpha_{n-r} \end{pmatrix}=0$.

\medskip

So

$$K=\begin{pmatrix} 0, I_{r_1} \end{pmatrix} \Phi^{-1}\begin{pmatrix} \alpha_{q_1r} &\cdots& \alpha_{n-1} \\ \vdots & \vdots & \vdots \\ \alpha_{(q_1-1)r+1} &\cdots &\alpha_{n-r} \end{pmatrix}=-\begin{pmatrix} 0, I_{r_1} \end{pmatrix} \begin{pmatrix} \beta_{q_1r} &\cdots& \beta_{n-1} \\ \vdots & \vdots & \vdots \\ \beta_{(q_1-1)r+1} &\cdots &\beta_{n-r} \end{pmatrix} $$

$$\begin{pmatrix} 1 &\cdots& 0 \\ \vdots & \ddots & \vdots \\ \alpha_{n-r_1+1} &\cdots &1 \end{pmatrix}=-\begin{pmatrix} \beta_{(q_1-1)r+r_1} &\cdots& \beta_{(q_1-1)r+2r_1-1} \\ \vdots & \vdots & \vdots \\ \beta_{(q_1-1)r+1} &\cdots &\beta_{(q_1-1)r+r_1}\end{pmatrix} \begin{pmatrix} 1 &\cdots& 0 \\ \vdots & \ddots & \vdots \\ \alpha_{n-r_1+1} &\cdots &1 \end{pmatrix}$$

Because each row in $K$ can be generated by the corresponding row in 

$\begin{pmatrix} \beta_{(q_1-1)r+r_1} & \cdots & \beta_{(q_1-1)r+2r_1-1} \\ \vdots & \vdots & \vdots \\ \beta_{(q_1-1)r+1} & \cdots & \beta_{(q_1-1)r+r_1} \end{pmatrix}$ and vice versa, to prove Lemma $3.9$, it is sufficient to show that $\beta_{(q_1-1)r+r_1}, \cdots,  \beta_{(q_1-1)r+2r_1-1} $ and $\psi_0, \cdots, \psi_{q_1r-1}$ generate all $\beta_{(q_1-1)r+i}$ for $1 \le i \le r_1-1$.

Denote $B^{-(q_1-1)} \widehat{D}=B^{-q_1} \widehat{A}$ as $I_n+\gamma_{n-1}L_n+\cdots+\gamma_{1} L_n^{n-1}$, where $\widehat{A}=B\widehat{D}$. From Equation $(3.7)$, we have that $\gamma_{(q_1-1)r+i} \propto \frac{\partial^{q_1-1} d_{r-1}}{\partial^{q_1-2} b_0 \partial b_{r-1-i}}=\psi_{(q_1-1)r+i}$ for $i=0, \cdots, r-1$.

\medskip

From the equation $(I_n+\alpha_{n-1}L_n+\cdots+\alpha_{1} L_n^{n-1})^{-1}=B^{q_1} \widehat{D}^{-1}$, we have that $B^{-(q_1-1)} \widehat{D}(I_n$

\noindent $+\alpha_{n-1}L_n+\cdots+\alpha_{1} L_n^{n-1})^{-1} =B$,  that is,

\begin{equation}
\begin{split}
(I_n+\gamma_{n-1}L_n+\cdots+\gamma_{1} L_n^{n-1}) (I_n+\beta_{n-1}L_n+\cdots+\beta_{1} L_n^{n-1})=I_n+b_{r-1}L_n+\cdots+b_0 L_n^r
\end{split}
\end{equation}

Comparing the coefficients of $L_n^k$ for $k=r+1, \cdots, r+r_1-1$ in both sides of Equation $(3.17)$, we have that 
\begin{equation}
\begin{split}
\beta_{(q_1-1)r+i}+\beta_{(q_1-1)r+i+1} \gamma_{n-1}+\cdots+\beta_{n-1} \gamma_{(q_1-1)r+i+1}+\gamma_{(q_1-1)r+i}=0 \hskip .1 cm for \hskip .1 cm i=r_1-1, \cdots, 1.
\end{split}
\end{equation}

For each term $\beta_{k} \gamma_{l}$ in the equation $\beta_{(q_1-1)r+r_1-1}+\beta_{(q_1-1)r+r_1} \gamma_{n-1}+\cdots+\beta_{n-1} \gamma_{(q_1-1)r+r_1}+\gamma_{(q_1-1)r+r_1-1}=0$, we have either $(q_1-1)r+r_1 \le k \le (q_1-1)r+2r_1-1$ or $(q_1-1)r+r_1 \le l=n+(q_1-1)r+r_1-1-k \le n+(q_1-1)r+r_1-1-(q_1-1)r-2r_1=q_1r-1=(q_1-1)r+r-1$. So $\beta_{(q_1-1)r+r_1-1}$ can be generated by $\beta_{(q_1-1)r+r_1}, \cdots,  \beta_{(q_1-1)r+2r_1-1}$ and $\gamma_{(q_1-1)r+r_1}=\psi_{(q_1-1)r+r_1}, \cdots, \gamma_{(q_1-1)r+r-1}=\psi_{(q_1-1)r+r-1}$. Using Equation $(3.18)$ and induction on $i$ in descend order, we can prove that $\beta_{(q_1-1)r+r_1}, \cdots,  \beta_{(q_1-1)r+2r_1-1}$ and $\gamma_{(q_1-1)r+1}=\psi_{(q_1-1)r+1}, \cdots, \gamma_{(q_1-1)r+r-1}=\psi_{(q_1-1)r+r-1}$ generate all $\beta_{(q_1-1)r+i}$ for $1 \le i \le r_1-1$.

This completes the proof of Lemma $3.9$.

\end{proof}

Let $\psi_{q_1r+i}=\beta_{(q_1-1)r+r_1+i}$ for $i=0, \cdots, r_1-1$. We have

\begin{thm}
$\Delta^{r_1}(\overset{q_1}{\overbrace{\Delta^r \cdots \Delta^r}} I(\mu_{n,r}))= I(\mu_{n,r})+(\psi_0, \cdots, \psi_{q_1r-1}, \psi_{q_1r}, \cdots, \psi_{q_1r+r_1-1})$
\end{thm}

\begin{proof} By Lemma $3.5$, Theorem $3.6$ and Proposition $3.8$, we only need to prove that any $(r-r_1+1) \times 
(r-r_1+1)$ minor of $\begin{pmatrix}\frac{\partial \psi_{(q_1-1)r+i}}{\partial b_j}\end{pmatrix}$ can be generated by $\psi_0, \cdots, \psi_{q_1r-1}$ and $\psi_{q_1r}, \cdots, \psi_{q_1r+r_1-1}$.

It is easy to deduce that $\begin{pmatrix}\frac{\partial \psi_{(q_1-1)r+i}}{\partial b_j}\end{pmatrix}=\begin{pmatrix}\frac{\partial^{q_1} d_i}{\partial^{q_1-1} b_0 \partial b_j}\end{pmatrix}=\begin{pmatrix}0&I_r\end{pmatrix} (-1)^{q_1} q_1! B^{-q_1} \left(\frac{\partial B}{\partial b_0} \right)^{q_1-1}$

\noindent $D=\begin{pmatrix}0&I_r\end{pmatrix} (-1)^{q_1} q_1! B^{-q_1} L_n^{(q_1-1)r} \widehat{D} \begin{pmatrix} I_r\\ 0 \end{pmatrix} \propto \begin{pmatrix}0&I_r\end{pmatrix} B^{-q_1} L_n^{(q_1-1)r} \widehat{D} \begin{pmatrix} I_r\\ 0 \end{pmatrix}$. So we only need to prove that any $(r-r_1+1) \times 
(r-r_1+1)$ minor of $\begin{pmatrix}0&I_r\end{pmatrix} B^{-q_1} L_n^{(q_1-1)r} \widehat{D} \begin{pmatrix} I_r\\ 0 \end{pmatrix}$, and hence $(I_r+\alpha_{n-1}L_r+\cdots+\alpha_{n-r+1} L_r^{r-1})^{-1} \begin{pmatrix} 0, I_r \end{pmatrix} L_n^{(q_1-1)r} B^{-q_1}\widehat{D} \begin{pmatrix} I_r\\ 0 \end{pmatrix}=\begin{pmatrix} * &I_{r-r_1}\\ K & 0 \end{pmatrix}$, can be generated by $\psi_0, \cdots, \psi_{q_1r-1}$ and $\psi_{q_1r}, \cdots, \psi_{q_1r+r_1-1}$. Any $(r-r_1+1) \times (r-r_1+1)$ minor of $\begin{pmatrix} * &I_{r-r_1}\\ K & 0 \end{pmatrix}$ must contain a row with elements either in $K$ or equal to zero. Expanding this minor along that row gives that elements in $K$ generate the minor. By the proof of Lemma $3.9$, each element in $K$ can be generated by $\psi_0, \cdots, \psi_{q_1r-1}$ and $\beta_{(q_1-1)r+r_1}, \cdots, \beta_{(q_1-1)r+2 r_1-1}$. By the definition of $(\psi_{q_1r}, \cdots, \psi_{q_1r+r_1-1})$, we have that $\psi_0, \cdots, \psi_{q_1r-1}$ and $\psi_{q_1r}, \cdots, \psi_{q_1r+r_1-1}$ generate all $(r-r_1+1) \times (r-r_1+1)$ minors of $\begin{pmatrix}\frac{\partial \psi_{(q_1-1)r+i}}{\partial b_j}\end{pmatrix}$. This complete the proof of Theorem $3.10$.

\end{proof}

Denote $f_0(x)=f(x)$, $f_1(x)=g(x)$, $h_0(x)=h(x)$, $r_{-1}=n$ and $r_0=r$. We will show that a sequence of monic polynomials $f_0(x), f_1(x), f_2(x), \cdots, f_{k+1}(x)$ can be produced inductively starting from $f_0(x)$ and $f_1(x)$ such that the degree of $f_i(x)$ is $r_{i-1}$ and each product of $h_i(x)=f_i(x)f_{i+1}(x)$ gives a map $\mu_{r_{i-1},r_i}: \C^{r_{i-1}} \times \C^{r_i} \rightarrow \C^{r_{i-1}+r_i}$ with the property that the polynomials generated at each of the first $(q_{i+1}+1)$ steps of the critical extensions of $I(\mu_{r_{i-1},r_i})$ can be added into the corresponding steps to form the critical extensions of $I(\mu_{n,r})$.  

Recall that $B^{-q_1} \widehat{A}=I_n+\gamma_{n-1}L_n+\cdots+\gamma_{1} L_n^{n-1}$. Let $f_2(x)=x^{r_1}+\gamma_{n-1} x^{r_1-1}+\cdots+\gamma_{n-r_1}$. Then the product $h_1(x)= f_1(x) f_2(x)= g(x)f_2(x)=x^n+\sigma_{r+r_1-1} x^{n-1}+\cdots+\sigma_0$ gives a map $\mu_{r,r_1}: \C^r \times \C^{r_1} \rightarrow \C^{r+r_1}$. Taking derivatives of the coefficients of $h_1(x)$ with respect to $b_{r-1}, \cdots, b_0, \gamma_{n-1}, \cdots, \gamma_{n-r_1}$ gives its first Jacobian

\begin{equation}
\begin{split}
\delta I(\mu_{r,r_1})= \begin{pmatrix} \begin{pmatrix} I_{r+r_1}, 0 \end{pmatrix} \Gamma_{r+r_1+1} \begin{pmatrix} I_r\cr 0 \cr \end{pmatrix},\begin{pmatrix} I_{r+r_1}, 0 \end{pmatrix} B_{r+r_1+1} \begin{pmatrix} I_{r_1}\cr 0 \cr \end{pmatrix}\end{pmatrix}\end{split}
\end{equation}

\noindent where $\Gamma_{r+r_1+1}=I_{r+r_1+1}+\gamma_{n-1} L_{r+r_1+1}+\cdots+\gamma_{n-r_1} L_{r+r_1+1}^{r_1}$ and $B_{r+r_1+1}=I_{r+r_1+1}+b_{r-1} L_{r+r_1+1}+\cdots+b_0 L_{r+r_1+1}^{r}$.

Repeating the same process as we did for $I(\mu_{n,r})$, we get polynomials $\varphi_0, \cdots, \varphi_{r_1-1}, \cdots,$

\noindent $\varphi_{(q_2-1)r_1},\cdots, \varphi_{q_2r_1-1}$ and $\varphi_{q_2r_1},\cdots, \varphi_{q_2r_1+r_2-1}$ which satisfy

\begin{equation}
\begin{split}
\begin{pmatrix} \frac{\partial \varphi_{sr_1+i}}{\partial \gamma_{n-j}}\end{pmatrix}=- (s+1) \begin{pmatrix} \frac{\partial \varphi_{sr_1+i}}{\partial b_j} \end{pmatrix} \begin{pmatrix} I_{r}, 0 \end{pmatrix}\Gamma_{r+r_1}^{-1} \begin{pmatrix} I_{r+r_1}, 0 \end{pmatrix} B_{r+r_1+1} \begin{pmatrix} I_{r_1}\cr 0 \cr \end{pmatrix}  \text{for} \hskip .1 cm s=0, \cdots, \\q_2-1, i=0, \cdots, r_1-1 \hskip .1 cm \text{and} \hskip .1 cm j=1, \cdots, r_1, \text{where} \hskip .1 cm \Gamma_{r+r_1}=\begin{pmatrix} I_{r+r_1}, 0 \end{pmatrix} \Gamma_{r+r_1+1} \begin{pmatrix} I_{r+r_1}\cr 0 \cr \end{pmatrix}
\end{split}
\end{equation}

\noindent and

\begin{equation}
\begin{split}\overset{s}{\overbrace{\Delta^{r_1} \cdots \Delta^{r_1}}} I(\mu_{r,r_1})= I(\mu_{r,r_1})+(\varphi_0, \cdots, \varphi_{sr_1-1})\hskip .1 cm \text{for} \hskip .1 cm s=1, \cdots, q_2 \\
\Delta^{r_2}(\overset{q_2}{\overbrace{\Delta^{r_1} \cdots \Delta^{r_1}}} I(\mu_{r,r_1}))= I(\mu_{r,r_1})+(\varphi_0, \cdots, \varphi_{q_2r_1-1}, \varphi_{q_2r_1}, \cdots, \varphi_{q_2r_1+r_2-1})
\end{split}
\end{equation}

We will prove that adding these polynomials correspondingly into the generator sets gives the critical extensions of $\overset{q_1}{\overbrace{\Delta^{r} \cdots \Delta^{r}}} I(\mu_{n,r})$.

The following lemma is true.

\begin{lem} For any $s \hskip .1 cm (1 \le s \le q_1)$, the coefficients of $L_{n+r+1}^{i}$ in $B_{n+r+1}^{-s} A_{n+r+1}$ for $i=n-sr+1, \cdots, n+r$ are zeros $\mod (\psi_0, \cdots, \psi_{sr-1})$. 
\end{lem}

\begin{proof}
The case $s=1$ was proved at the beginning of this subsection. 

Suppose that we proved Lemma $3.11$ for $s \le p$. If $p=q_1$, we are done. Otherwise, we may assume that $1 \le p<q_1$. We will show that the coefficients of $L_{n+r+1}^{i}$ in $B_{n+r+1}^{-(p+1)} A_{n+r+1}$ for $i=n-(p+1)r+1, \cdots, n+r$ are zeros $\mod (\psi_0, \cdots, \psi_{(p+1)r-1})$.

Denote $B_{n+r+1}^{-(p+1)} A_{n+r+1}=I_{n+r+1}+\lambda_{n-1} L_{n+r+1}+\cdots+\lambda_{0} L_{n+r+1}^n+\lambda_{-1} L_{n+r+1}^{n+1}+\cdots+\lambda_{-r} L_{n+r+1}^{n+r}$ and $B_{n+r+1}^{-p} A_{n+r+1}=I_{n+r+1}+\kappa_{n-1} L_{n+r+1}+\cdots+\kappa_{0} L_{n+r+1}^n+\kappa_{-1} L_{n+r+1}^{n+1}+\cdots+\kappa_{-r} L_{n+r+1}^{n+r}$. By inductive assumption $\kappa_{n-i} \equiv 0 \mod (\psi_0, \cdots, \psi_{pr-1})$ for $i=n-pr+1, \cdots, n+r$. 

By Equation $(3.7)$, we have $\begin{pmatrix} I_n, 0 \end{pmatrix} (-1)^p p! B_{n+r+1}^{-(p+1)} L_{n+r+1}^{(p-1)r} A_{n+r+1} \begin{pmatrix} I_r \\ 0 \end{pmatrix}=(-1)^p p! B^{-p-1} $

\noindent $L_n^{(p-1)r}A=\begin{pmatrix} \frac{\partial^{p} d_{i}}{\partial^{p-1} b_0 \partial b_j} \end{pmatrix}$. Left multiplying this equation by $e_{n-r+1}$ gives that $\lambda_{pr+i} \propto \frac{\partial^{p} d_{r-1}}{\partial^{p-1} b_0 \partial b_{r-1-i}}=\psi_{pr+i}$ for $i=0, \cdots, r-1$, where $e_{n-r+1}$ is a $1 \times n$ vector with $1$ in the $(n-r+1)$ position and zero elsewhere. Comparing the coefficients of $L_{n+r+1}^i$ in both sides of the equation $(B_{n+r+1}^{-(p+1)} A_{n+r+1}) B_{n+r+1}=B_{n+r+1}^{-p} A_{n+r+1}$ for $i=n-pr+1, \cdots, n+r$, we have that 
{\small
\begin{equation}
\begin{split}
\lambda_{n-i}+\lambda_{n-i+1} b_{r-1}+\cdots+\lambda_{n-i+r} b_0 \equiv 0  \mod (\psi_0, \cdots, \psi_{(p+1)r-1}) \hskip .1 cm \text{for} \hskip .1 cm i=n-pr+1, \cdots, n+r
\end{split}
\end{equation}
}
Using Equation $(3.22)$ and $\lambda_{pr+i} \propto \psi_{pr+i}$ for $i=0, \cdots, r-1$, we immediately have that $\lambda_{n-i} \equiv 0 \mod (\psi_0, \cdots, \psi_{(p+1)r-1}) \hskip .1 cm \text{for} \hskip .1 cm i=n-pr+1, \cdots, n+r$. Because $\lambda_{pr+i}=\lambda_{n-(n-pr-i)}$ for $i=0, \cdots, r-1$, so $\lambda_{n-i} \equiv 0 \mod (\psi_0, \cdots, \psi_{(p+1)r-1}) \hskip .1 cm \text{for} \hskip .1 cm i=n-(p+1)r+1, \cdots, n-pr$ as well. This completes the proof of Lemma $3.11$.

\end{proof}

From Lemma $3.11$, we have the following proposition.

\begin{prop}
$\varphi_i \equiv \psi_{q_1r+i} \mod (\psi_0, \cdots, \psi_{q_1r-1})$ for $i=0, \cdots, r_1-1$.
\end{prop}

\begin{proof}

From the first Jacobian $\delta I(\mu_{r,r_1})$, we have that $\Gamma_{r+r_1}^{-1}\delta I(\mu_{r,r_1})=\begin{pmatrix} \begin{pmatrix}I_r\\0\end{pmatrix} , W \begin{pmatrix}I_{r_1}\\0\end{pmatrix} \end{pmatrix}$, where 
$W=\Gamma_{r+r_1}^{-1} \begin{pmatrix} I_{r+r_1}, 0 \end{pmatrix} B_{r+r_1+1} \begin{pmatrix} I_{r+r_1}\cr 0 \cr \end{pmatrix}$. Denote $W$ as $I_{r+r_1}+ w_{r-1} L_{r+r_1}+\cdots+ w_0 L_{r+r_1}^{r}+w_{-1} L_{r+r_1}^{r+1}+ \cdots+ w_{-r_1+1} L_{r+r_1}^{r+r_1-1}$. We have that $w_i=\varphi_i$ for $i=0, \cdots, r_1-1$.

By Lemma $3.11$, $\begin{pmatrix}I_{r+r_1}, 0\end{pmatrix} B_{n+r+1}^{-q_1} A_{n+r+1} \begin{pmatrix}I_{r+r_1}\\ 0\end{pmatrix} \equiv I_{r+r_1}+ \gamma_{n-1} L_{r+r_1}+\cdots+\gamma_{n-r_1} L_{r+r_1}^{r_1}$

\noindent $\equiv I_{r+r_1}+ \gamma_{n-1} L_{r+r_1}+\cdots+\gamma_{n-r} L_{r+r_1}^{r} \mod (\psi_0, \cdots, \psi_{q_1r-1})$.
Comparing the coefficients of $L_n^k$ in both sides of Equation $(3.17)$ and $L_{r+r_1}^k$ in both sides of $(I_{r+r_1}+ \gamma_{n-1} L_{r+r_1}+\cdots+\gamma_{n-r} L_{r+r_1}^{r}) (I_{r+r_1}+ w_{r-1} L_{r+r_1}+\cdots+ w_0 L_{r+r_1}^{r}+w_{-1} L_{r+r_1}^{r+1}+ \cdots+ w_{-r_1+1} L_{r+r_1}^{r+r_1-1}) \equiv \Gamma_{r+r_1} W=\begin{pmatrix} I_{r+r_1}, 0 \end{pmatrix} B_{r+r_1+1} \begin{pmatrix} I_{r+r_1}\cr 0 \cr \end{pmatrix}=I_{r+r_1}+b_{r-1} L_{r+r_1}+ \cdots + b_0 L_{r+r_1}^r \mod (\psi_0, \cdots, \psi_{q_1r-1})$ for $k=1, \cdots, r$, we have that

\begin{equation}
\begin{split}
\beta_{n-k}+ \beta_{n-k+1} \gamma_{n-1}+\cdots+ \beta_{n-1} \gamma_{n-k+1}+ \gamma_{n-k}=b_{r-k}\equiv w_{r-k}+w_{r-k+1} \gamma_{n-1}+\\ \cdots + w_{r-1} \gamma_{n-k+1}+\gamma_{n-k} \mod (\psi_0, \cdots, \psi_{q_1r-1}) \hskip .1 cm for \hskip .1 cm k=1, \cdots, r \\ 
\end{split}
\end{equation}

Let $k=1$ in Equation $(3.23)$, we have $\beta_{n-1}+\gamma_{n-1} \equiv w_{r-1}+\gamma_{n-1}$, so $\beta_{n-1} \equiv w_{r-1}$. Let $k=2, \cdots, r$ in Equation $(3.23)$ and use induction, we have $\beta_{n-k} \equiv w_{r-k}$ for each $k=1, \cdots, r$, so $\varphi_i=w_i \equiv \beta_{n-r+i}=\beta_{(q_1-1)r+r_1+i}=\psi_{q_1r+i} \mod (\psi_0, \cdots, \psi_{q_1r-1})$ for $i=0, \cdots, r_1-1$.
\end{proof}

As an immediate corollary of Theorem $3.10$ and Proposition $3.12$, we have

\begin{cor}
$\Delta^{r_1}(\overset{q_1}{\overbrace{\Delta^r \cdots \Delta^r}} I(\mu_{n,r}))=I(\mu_{n,r})+(\psi_0, \cdots, \psi_{q_1r-1}, \varphi_{0}, \cdots, \varphi_{r_1-1})$
\end{cor}

Now we can prove the following theorem.

\begin{thm}
$\overset{s}{\overbrace{\Delta^{r_1} \cdots \Delta^{r_1}}} \overset{q_1}{\overbrace{\Delta^r \cdots \Delta^r}} I(\mu_{n,r})= I(\mu_{n,r})+(\psi_0, \cdots, \psi_{q_1r-1}, \varphi_0, \cdots, \varphi_{sr_1-1})$ 

for $s=1, \cdots, q_2$ and

$\Delta^{r_2}(\overset{q_2}{\overbrace{\Delta^{r_1} \cdots \Delta^{r_1}}} \overset{q_1}{\overbrace{\Delta^r \cdots \Delta^r}} I(\mu_{n,r})= I(\mu_{n,r})+(\psi_0, \cdots, \psi_{q_1r-1},\varphi_0, \cdots, \varphi_{q_2r_1-1}, \varphi_{q_2r_1},$

$ \cdots, \varphi_{q_2r_1+r_2-1})$
\end{thm}

We need the following two lemmas.

\begin{lem} $(\psi_0, \cdots, \psi_{q_1r-1}, \gamma_{n-r_1}, \cdots, \gamma_{n-1}, b_0, \cdots, b_{r-1})$ forms a new local coordinate system around zero on $\C^n \times \C^r$.
\end{lem}
 
\begin{proof}

Let $A_{n+1}=\begin{pmatrix}I_{n+1}, 0\end{pmatrix} A_{n+r+1} \begin{pmatrix}I_{n+1}\\0\end{pmatrix}$ and $B_{n+1}=\begin{pmatrix}I_{n+1}, 0\end{pmatrix} B_{n+r+1} \begin{pmatrix}I_{n+1}\\0\end{pmatrix}$. Then we have that $A_{n+1}=I_{n+1}+a_{n-1} L_{n+1}+\cdots+a_0 L_{n+1}^n$, $B_{n+1}=I_{n+1}+b_{r-1} L_{n+1}+\cdots+b_0 L_{n+1}^r$ and $B_{n+1}^{-q_1}A_{n+1}=I_{n+1}+\gamma_{n-1} L_{n+1}+ \cdots+ \gamma_0 L_{n+1}^n$. By Lemma $3.1 (2)$ and the equation $A_{n+1}= (B_{n+1}^{-q_1}A_{n+1}) B_{n+1}^{q_1}$, it is easy to see that $\gamma_0, \cdots, \gamma_{n-1}, b_0, \cdots, b_{r-1}$ form a new local coordinate system around zero on $\C^n \times \C^r$.

By Lemma $3.11$, $\gamma_i \equiv 0 \mod (\psi_0, \cdots, \psi_{q_1r-1})$ for $i=0, \cdots, q_1 r-1$. So to prove Lemma $3.15$, it is sufficient to prove that $\psi_0, \cdots, \psi_{q_1r-1}$ are polynomials in $\gamma_0, \cdots, \gamma_{n-1}, b_0, \cdots, b_{r-1}$. This follows directly from the fact that $\psi_i$s are polynomials in $a_0, \cdots, a_{n-1}, b_0, \cdots, b_{r-1}$ while $a_0, \cdots, a_{n-1}, b_0, \cdots, b_{r-1}$ are polynomials in $\gamma_0, \cdots, \gamma_{n-1}, b_0, \cdots, b_{r-1}$. 
\end{proof}

Denote $h_1(x) g(x)^{q_1}=x^{n+r}+\tau_{n+r-1} x^{n+r-1}+\cdots+\tau_0$. We have

\begin{lem}
$c_i \equiv \tau_i \mod (\psi_0, \cdots, \psi_{q_1r-1})$ for $i=0, \cdots, n+r-1$.
\end{lem}
\begin{proof}

Let $C_{n+r+1}=I_{n+r+1}+c_{n+r-1} L_{n+r+1}+\cdots+c_0 L_{n+r+1}^{n+r}$, where $c_i$s are the coefficients of $h(x)$. Then $C_{n+r+1}=A_{n+r+1} B_{n+r+1}$. By Lemma $3.11$, $B_{n+r+1}^{-q_1} A_{n+r+1} \equiv I_{n+r+1}+\gamma_{n-1} L_{n+r+1}+\cdots+\gamma_{n-r_1} L_{n+r+1}^{r_1} \mod (\psi_0, \cdots, \psi_{q_1r-1})$. So $C_{n+r+1}=A_{n+r+1} B_{n+r+1}=(B_{n+r+1}^{-q_1} A_{n+r+1})B_{n+r+1}^{q_1+1} \equiv (I_{n+r+1}+\gamma_{n-1} L_{n+r+1}+\cdots+\gamma_{n-r_1} L_{n+r+1}^{r_1}) B_{n+r+1}^{q_1+1}=I_{n+r+1}+\tau_{n+r-1} L_{n+r+1}+\cdots+ \tau_0 L_{n+r+1}^{n+r} \mod (\psi_0, \cdots, \psi_{q_1r-1})$. Therefore $c_i \equiv \tau_i \mod (\psi_0, \cdots, \psi_{q_1r-1})$ for $i=0, \cdots, n+r-1$.

\end{proof}

Proof of Theorem $3.14$:

\begin{proof} The case $s=1$ is Corollary $3.13$.

Suppose we have proved that 

$\overset{s}{\overbrace{\Delta^{r_1} \cdots \Delta^{r_1}}} \overset{q_1}{\overbrace{\Delta^r \cdots \Delta^r}} I(\mu_{n,r})= I(\mu_{n,r})+(\psi_0, \cdots, \psi_{q_1r-1}, \varphi_0, \cdots, \varphi_{sr_1-1})$ for $s=1, \cdots, p$

If $p=q_2$, we have done the first part of Theorem $3.14$. Otherwise, we may assume $p<q_2$. We will show that 

$\overset{p+1}{\overbrace{\Delta^{r_1} \cdots \Delta^{r_1}}} \overset{q_1}{\overbrace{\Delta^r \cdots \Delta^r}} I(\mu_{n,r})= I(\mu_{n,r})+(\psi_0, \cdots, \psi_{q_1r-1}, \varphi_0, \cdots, \varphi_{(p+1)r_1-1})$

By Lemma $3.16$, $\overset{p}{\overbrace{\Delta^{r_1} \cdots \Delta^{r_1}}} \overset{q_1}{\overbrace{\Delta^r \cdots \Delta^r}} I(\mu_{n,r})=I(\mu_{n,r})+(\psi_0, \cdots, \psi_{q_1r-1}, \varphi_0, \cdots,\varphi_{pr_1-1})$

\noindent $= (c_{n+r-1}, \cdots, c_0, \psi_0, \cdots, \psi_{q_1r-1}, \varphi_0, \cdots, \varphi_{pr_1-1}) = (\tau_{n+r-1}, \cdots, \tau_0, \psi_0, \cdots, \psi_{q_1r-1}, \varphi_0, \cdots, $

\noindent $\varphi_{pr_1-1})$. Considering $\delta \overset{p}{\overbrace{\Delta^{r_1} \cdots \Delta^{r_1}}} \overset{q_1}{\overbrace{\Delta^r \cdots \Delta^r}} I(\mu_{n,r})=\delta (\tau_{n+r-1}, \cdots, \tau_0, \psi_0, \cdots, \psi_{q_1r-1}, \varphi_0, \cdots$

\noindent $,\varphi_{pr_1-1})$ with derivatives taken with respect to the new coordinate system in the order $\gamma_{n-1}, \cdots, \gamma_{n-r_1},\psi_{q_1r-1}, \cdots, \psi_0, b_{r-1}, \cdots, b_0$, we have that it is given by the following matrix

\begin{equation}
\begin{split}
\begin{pmatrix} 
B_{n+r}^{q_1+1} \begin{pmatrix} I_{r_1} \cr 0 \cr \end{pmatrix}& 0 & (q_1+1) B_{n+r}^{q_1}\Gamma_{n+r} \begin{pmatrix} I_r\cr 0 \cr \end{pmatrix}\\
0&I_{q_1r}&0\\
\begin{pmatrix} \frac{\partial \varphi_{sr_1+i}}{\partial \gamma_{n-j}}\end{pmatrix}&0& \begin{pmatrix} \frac{\partial \varphi_{sr_1+i}}{\partial b_j} \end{pmatrix}\\ 
\end{pmatrix}
\end{split}
\end{equation}

where $\Gamma_{n+r}=I_{n+r}+\gamma_{n-1} L_{n+r}+\cdots+\gamma_{n-r_1} L_{n+r}^{r_1}$ and $s$ varies from $0$ to $p-1$.

Left multiplying matrix $(3.24)$ by $\begin{pmatrix} B_{n+r}^{-q_1}& 0 \\ 0& I_{q_1r+pr_1} \\ \end{pmatrix}$, we have

{\footnotesize \begin{equation}
\begin{split}
\begin{pmatrix} 
B_{n+r} \begin{pmatrix} I_{r_1} \cr 0 \cr \end{pmatrix}& 0 & (q_1+1) \Gamma_{n+r} \begin{pmatrix} I_r\cr 0 \cr \end{pmatrix}\\
0&I_{q_1r}&0\\
\begin{pmatrix} \frac{\partial \varphi_{sr_1+i}}{\partial \gamma_{n-j}}\end{pmatrix}&0& \begin{pmatrix} \frac{\partial \varphi_{sr_1+i}}{\partial b_j} \end{pmatrix}\\ 
\end{pmatrix}=\begin{pmatrix} 
\begin{pmatrix} I_{r+r_1}, 0 \end{pmatrix} B_{r+r_1+1} \begin{pmatrix} I_{r_1} \cr 0 \cr \end{pmatrix}& 0 & (q_1+1) \Gamma_{r+r_1} \begin{pmatrix} I_r\cr 0 \cr \end{pmatrix}\\
0&0&0\\
0&I_{q_1r}&0\\
\begin{pmatrix} \frac{\partial \varphi_{sr_1+i}}{\partial \gamma_{n-j}}\end{pmatrix}&0& \begin{pmatrix} \frac{\partial \varphi_{sr_1+i}}{\partial b_j} \end{pmatrix}\\ 
\end{pmatrix}
\end{split}
\end{equation}}

Left multiplying $(3.25)$ by {\footnotesize $\begin{pmatrix} \frac{1}{q_1+1} \Gamma_{r+r_1}^{-1}& 0\\0&I_{2q_1r+pr_1} \end{pmatrix} \begin{pmatrix}I_{r+r_1}&0&0\\0&I_{2q_1r}&0\\\frac{-1}{q_1+1} \begin{pmatrix} \frac{\partial \varphi_{sr_1+i}}{\partial b_j} \end{pmatrix} \begin{pmatrix} I_{r}, 0 \end{pmatrix} \Gamma_{r+r_1}^{-1}&0&I_{pr_1}\\ \end{pmatrix}$} and using Equation $(3.20)$, we have that

{\small
\begin{equation}
\begin{split}
\begin{pmatrix} 
\frac{1}{q_1+1} \Gamma_{r+r_1}^{-1}\begin{pmatrix} I_{r+r_1}, 0 \end{pmatrix} B \begin{pmatrix} I_{r_1} \cr 0 \cr \end{pmatrix}& 0 & \begin{pmatrix} I_r\cr 0 \cr \end{pmatrix}\\
0&0&0\\
0&I_{q_1r}&0\\
(1+\frac{1}{(s+1)(q_1+1)}) \begin{pmatrix} \frac{\partial \varphi_{sr_1+i}}{\partial \gamma_{n-j}}\end{pmatrix}&0& 0  
\end{pmatrix}
\end{split}
\end{equation}
}

By induction and the construction of $\varphi_i$s, the matrix $(3.26)$ is equal to

{\small
\begin{equation}
\begin{split}
\begin{pmatrix} 
\frac{1}{q_1+1} \begin{pmatrix} I_{r}, 0 \end{pmatrix} \Gamma_{r+r_1}^{-1}\begin{pmatrix} I_{r+r_1}, 0 \end{pmatrix} B \begin{pmatrix} I_{r_1} \cr 0 \cr \end{pmatrix}& 0 & I_r \\
0&0&0\\
0&I_{q_1r}&0\\
0&0& 0  
\end{pmatrix} \mod (\varphi_0, \cdots, \varphi_{(p+1)r_1-1})
\end{split}
\end{equation}
} 
So the corank of $\delta \overset{p}{\overbrace{\Delta^{r_1} \cdots \Delta^{r_1}}} \overset{q_1}{\overbrace{\Delta^r \cdots \Delta^r}} I(\mu_{n,r})$ is $r_1$ and any $(n+r-r_1+1)$ minor is generated by $(\varphi_0, \cdots, \varphi_{(p+1)r_1-1})$. Because each $\varphi_i$ is only different from a $(n+r-r_1+1)$ minor of $\delta \overset{p}{\overbrace{\Delta^{r_1} \cdots \Delta^{r_1}}} \overset{q_1}{\overbrace{\Delta^r \cdots \Delta^r}} I(\mu_{n,r})$ by a nonzero constant, so

$\overset{p+1}{\overbrace{\Delta^{r_1} \cdots \Delta^{r_1}}} \overset{q_1}{\overbrace{\Delta^r \cdots \Delta^r}} I(\mu_{n,r})= I(\mu_{n,r})+(\psi_0, \cdots, \psi_{q_1r-1}, \varphi_0, \cdots, \varphi_{(p+1)r_1-1})$

This completes the proof of the first part of Theorem $3.14$.

Using the same idea, we can prove that $\delta \overset{q_2}{\overbrace{\Delta^{r_1} \cdots \Delta^{r_1}}} \overset{q_1}{\overbrace{\Delta^r \cdots \Delta^r}} I(\mu_{n,r})$ is equivalent to

\begin{equation}
\begin{split}
\begin{pmatrix} 
\frac{1}{q_1+1} \begin{pmatrix} I_{r}, 0 \end{pmatrix} \Gamma_{r+r_1}^{-1}\begin{pmatrix} I_{r+r_1}, 0 \end{pmatrix} B \begin{pmatrix} I_{r_1} \cr 0 \cr \end{pmatrix}& 0 & I_r \\
0&0&0\\
0&I_{q_1r}&0\\
0&0& 0\\
(1+\frac{1}{q_2(q_1+1)}) \begin{pmatrix} \frac{\partial \varphi_{(q_2-1)r_1+i}}{\partial \gamma_{n-j}}\end{pmatrix}&0&0\\  
\end{pmatrix} \mod (\varphi_0, \cdots, \varphi_{q_2r_1-1})
\end{split}
\end{equation}

By comparison with $\delta \overset{q_2}{\overbrace{\Delta^{r_1} \cdots \Delta^{r_1}}} I(\mu_{r,r_1})$ which is equivalent to

\begin{equation}
\begin{split}
\begin{pmatrix} 
\frac{1}{q_1+1} \begin{pmatrix} I_{r}, 0 \end{pmatrix} \Gamma_{r+r_1}^{-1}\begin{pmatrix} I_{r+r_1}, 0 \end{pmatrix} B \begin{pmatrix} I_{r_1} \cr 0 \cr \end{pmatrix}& I_r \\
0&0\\
(1+\frac{1}{q_2}) \begin{pmatrix} \frac{\partial \varphi_{(q_2-1)r_1+i}}{\partial \gamma_{n-j}}\end{pmatrix}&0\\  
\end{pmatrix} \mod (\varphi_0, \cdots, \varphi_{q_2r_1-1})
\end{split}
\end{equation}
the  matrices $(3.28)$ and $(3.29)$ have the same corank. By induction, the latter one has corank $r_2$, so does the matrix $(3.28)$. All the $(n+r-r_2+1)$ minors of matrix $(3.28)$ and the $(r+r_1-r_2+1)$ minors of matrix $(3.29)$ generate the same idea, so
 
 $\Delta^{r_2}(\overset{q_2}{\overbrace{\Delta^{r_1} \cdots \Delta^{r_1}}} \overset{q_1}{\overbrace{\Delta^r \cdots \Delta^r}} I(\mu_{n,r})= I(\mu_{n,r})+(\psi_0, \cdots, \psi_{q_1r-1},\varphi_0, \cdots, \varphi_{q_2r_1-1}, \varphi_{q_2r_1},$

$ \cdots, \varphi_{q_2r_1+r_2-1})$.
\end{proof}  

As a consequence of Theorem $3.14$, we have that
\begin{cor}
The first $(q_1+q_2+1)$ entries in $TB(I(\mu_{n,r}))$ is $(r, \cdots, r, r_1, \cdots, r_1, r_2)$ with $r$ repeating $q_1$ times and $r_1$ repeating $q_2$ times.
\end{cor}

Replacing $f_0(x), f_1(x)$ with $f_1(x), f_2(x)$ and repeating the same process, we can produce $f_3(x)$ of degree $r_2$ and a map $\mu_{r_1,r_2}: \C^{r_1} \times \C^{r_2} \rightarrow \C^{r_1+r_2}$. Using the first $(q_3+1)$ critical extensions of $I(\mu_{r_1,r_2})$, we can generate polynomials $\varphi_{q_2r_1},\cdots, \varphi_{q_2r_1+q_3 r_2-1}, \varphi_{q_2r_1+q_3 r_2}, \cdots $

\noindent $\varphi_{q_2r_1+q_3 r_2+r_3-1}$ such that they can be added correspondingly into the generator set to form the next $(q_3+1)$ critical extensions of $I(\mu_{n,r})$. Repeating the same procedure over and over, we can produce $f_4(x), \cdots, f_{k+1}(x)$ and use them to prove $TB(I(\mu_{n,r}))=I(n,r)$. Due to the heavy notations, we will not do so here. Instead, we mention a key observation that explains why we can add polynomials at each step to obtain the corresponding critical extension of $I(\mu_{n,r})$. At the $(q_1+\cdots+q_p+s)$-th step of the critical extension of $I(\mu_{n,r})$ for some $1 \le s \le q_{p+1}$, the Jacobian matrix is equivalent to a matrix with the form

{\small
\begin{equation}
\begin{split}
\begin{pmatrix} 
I_{r_{p-1}}& 0 & \mu \Lambda \\
0&0&0\\
0&I_{q_1r+q_2r_1+\cdots+q_p r_{p-1}}&0\\
0&0& 0\\
0&0&\nu \Theta \\  
\end{pmatrix} \mod (\varphi_{q_2+ \cdots +q_p}, \cdots, \varphi_{q_2+\cdots+q_p+s r_p-1})
\end{split}
\end{equation}
}
\noindent for some nonzero constant $\mu$ and $\nu$.

The matrix $\delta \overset{s}{\overbrace{\Delta^{r_{p-1}} \cdots \Delta^{r_{p-1}}}} I(\mu_{r_{p-1},r_p})$ is equivalent to

\begin{equation}
\begin{split}
\begin{pmatrix} 
I_{r_{p-1}}&  \Lambda\\
0&0\\
 0&\nu' \Theta\\  
\end{pmatrix} \mod (\varphi_{q_2+ \cdots +q_p}, \cdots, \varphi_{q_2+\cdots+q_p+s r_p-1})
\end{split}
\end{equation}

for some nonzero constant $\nu'$. 

It is obvious that the matrices $(3.30)$ and $(3.31)$ have the same corank and the corresponding minors generate the same ideal. So we can add the polynomials generated by the first $(q_{p+1}+1)$ critical extensions of $I(\mu_{r_{p-1},r_p})$ into the corresponding generator sets to form the critical extensions of $I(\mu_{n,r})$. Hence Theorem $1.2$ is true. This completes the proof of Varley's Conjecture.

\bigskip
\begin{center}  
{\bf References}
\end{center}
\bigskip

\begin{enumerate}
\item M. Adams, C. McCrory, T. Shifrin, R. Varley, Invariants of Gauss Maps of Theta Divisors, Proc. Sympos. Pure Math., Vol 54, pp 1-8, Amer. Math. Soc., Providence, RI, 1993.

\item V.I. Arnol$'$d, A.N. Varchenko, S.M. Gusein-Zade, Singularities of Differentiable Maps, Volume I, Birh\"auser, Boston, 1985.

\item J.M. Boardman, Singularities of Differential Maps, Institut des Hautes \'Etudes Scientifiques, Publications Mathematiques, vol 33, pp 21-57, 1967.

\item J. Wethington, On computing the Thom-Boardman symbols for polynomial multiplication maps, Ph.D. Dissertation at UGA, 2002.
\end{enumerate}
\medskip
{\small

Jiayuan Lin

\medskip
Affiliation: SUNY Canton

Mailing address:

Department of Mathematics

SUNY Canton

34 Cornell Drive

Canton, NY 13617

linj@canton.edu

\medskip

Janice Wethington

\medskip
Affiliation: U.S. Department of Defense

Mailing address:

1509 Stevens Creek Drive

North Augusta, SC 29860

janice1729@yahoo.com}

\end{document}